\documentclass[a4paper,12pt]{preprint}
\usepackage[margin=1.3in]{geometry}  
\usepackage[full]{textcomp}
\usepackage[osf]{newtxtext}

\usepackage{mathalfa}
\usepackage{microtype}

\usepackage{enumitem}
\usepackage{amsmath}
\usepackage{amsfonts} 
\usepackage{amssymb}             

\usepackage{comment}
\usepackage{breakurl}

\usepackage{mathrsfs}
\usepackage{booktabs}
\usepackage{mathtools}

\usepackage{tikz}
\usepackage{amsthm}                
\usepackage{mhequ}
\usepackage{hyperref}

\usepackage{accents}

\hypersetup{pdfstartview=}

\addtocontents{toc}{\protect\hypersetup{hidelinks}}

\DeclareSymbolFont{sfoperators}{OT1}{ptm}{m}{n}
\DeclareSymbolFontAlphabet{\mathsf}{sfoperators}

\makeatletter
\def\operator@font{\mathgroup\symsfoperators}
\makeatother
\newcommand{\eqdef}{\stackrel{\mbox{\tiny def}}{=}}

\numberwithin{equation}{section}

\newtheorem{thm}{Theorem}[section]

\newtheorem{lem}[thm]{Lemma}
\newtheorem{prop}[thm]{Proposition}

\newtheorem{cor}[thm]{Corollary}

\newtheorem{assumption}[thm]{Assumption}
\theoremstyle{remark}
\newtheorem{remark}[thm]{Remark}
\newtheorem{rmk}[thm]{Remark}

\makeatletter
\def\th@newremark{\th@remark\thm@headfont{\bfseries}}

\def\bdiamond{\mathop{\mathpalette\bdi@mond\relax}}
\newcommand\bdi@mond[2]{%
	\vcenter{\hbox{\m@th
			\scalebox{\ifx#1\displaystyle 2.6\else1.8\fi}{$#1\diamond$}%
	}}%
}

\def\bDiamond{\mathop{\mathpalette\bDi@mond\relax}}
\newcommand\bDi@mond[2]{%
	\vcenter{\hbox{\m@th
			\scalebox{\ifx#1\displaystyle 2.6\else1.2\fi}{$#1\Diamond$}%
	}}%
}
\makeatletter

\definecolor{darkgreen}{rgb}{0.1,0.7,0.1}
\definecolor{darkred}{rgb}{0.7,0.1,0.1}
\definecolor{darkblue}{rgb}{0,0,0.7}
\newcommand{\EE}{\mathbb{E}}
\newcommand{\NN}{\mathbb{N}}
\newcommand{\PP}{\mathbb{P}}
\newcommand{\RR}{\mathbb{R}}

\newcommand{\bB}{\mathcal{B}}

\newcommand{\dD}{\mathcal{D}}

\newcommand{\eE}{\mathcal{E}}
\newcommand{\fF}{\mathcal{F}}

\newcommand{\hH}{\mathcal{H}}
\newcommand{\iI}{\mathcal{I}}
\newcommand{\kK}{\mathcal{K}}
\newcommand{\mM}{\mathcal{M}}
\newcommand{\nN}{\mathcal{N}}

\newcommand{\R}{\mathcal{R}}
\newcommand{\sS}{\mathcal{S}}
\newcommand{\tT}{\mathcal{T}}

\newcommand{\xX}{\mathcal{X}}

\newcommand{\yY}{\mathcal{Y}}

\newcommand{\1}{\mathbf{1}}

\newcommand{\LLL}{\mathcal{L}}

\newcommand{\eps}{\varepsilon}

\newcommand{\FFF}{\mathcal{F}}
\newcommand{\uem}{u_{m,\eps}}
\newcommand{\vem}{v_{m,\eps}}

\makeatletter 
\newcommand{\cov}{{\operator@font cov}}
\newcommand{\var}{{\operator@font var}}
\newcommand{\corr}{{\operator@font corr}}
\newcommand{\diam}{{\operator@font diam}}
\newcommand{\Av}{{\operator@font Av}}
\newcommand{\trig}{{\operator@font trig}}
\newcommand{\Enh}{{\operator@font Enh}}
\makeatother

\newcommand{\lfl}{\left\lfloor }  
\newcommand{\rfl}{\right\rfloor} 

\colorlet{symbols}{blue!90!black}
\colorlet{testcolor}{green!60!black}

\def\${|\!|\!|}

\makeatletter
\def\DeclareSymbol#1#2#3{\expandafter\gdef\csname MH@symb@#1\endcsname{\tikz[baseline=#2,scale=0.15,draw=symbols]{#3}}\expandafter\gdef\csname MH@symb@#1s\endcsname{\scalebox{0.7}{\tikz[baseline=#2,scale=0.15,draw=symbols]{#3}}}}
\def\<#1>{\csname MH@symb@#1\endcsname}
\makeatother

\setlist[itemize]{topsep=3pt,itemsep=1.5pt,parsep=0pt}

\def\scal#1{\langle#1\rangle}
\def\cent#1{\mathopen{{\langle\kern-0.3em\rangle}}#1\mathclose{{\langle\kern-0.3em\rangle}}}

\def\d{\partial}

\begin{document}

\title{Global well-posedness for the defocusing mass-critical stochastic nonlinear Schr\"{o}dinger equation on $\RR$ at $L^{2}$ regularity}
\author{Chenjie Fan$^1$ and Weijun Xu$^2$}
\institute{University of Chicago, US, \email{cjfanpku@gmail.com}
\and University of Oxford, UK / NYU Shanghai, China, \email{weijunx@gmail.com}}

\maketitle

\begin{abstract}
	We prove global existence and stability of solution to the mass-critical stochastic nonlinear Schr\"{o}dinger equation in $d=1$ at $L^{2}$ regularity. Our construction starts with the existence of solution to the truncated subcritical problem. With the presence of truncation, we construct the solution to the critical equation as the limit of subcritical solutions. We then obtain uniform bounds on the solutions to the truncated critical problems that allow us to remove truncation in the limit. 
\end{abstract}

\setcounter{tocdepth}{2}
\microtypesetup{protrusion=false}
\tableofcontents
\microtypesetup{protrusion=true}

\section{Introduction}

\subsection{The problem and the main statement}

The aim of this article is to show global existence of the solution to the one dimensional stochastic nonlinear Schr\"{o}dinger equation (SNLS)
\begin{equation} \label{eq:main_eq}
i \partial_t u+\Delta  u = |u|^4 u + u \circ \dot{W}, \qquad x \in \RR,\; t \geq 0
\end{equation}
with arbitrary $L^2$ initial data. Here, $\dot{W}$ is a real-valued Gaussian process that is white in time and coloured in space. The precise assumption on the noise will be specified below. In the equation, $\circ$ denotes the Stratonovich product, which is the only choice of the product that preserves the $L^2$-norm of the solution. In practice, it is more convenient to treat the equation in the It\^{o} form. In order to be more precise about the equation and the It\^{o}-Stratonovich correction, we give the precise definition of the noise below. 

Let $\{B_k\}_{k \in \NN}$ be a sequence of independent standard Brownian motions defined on some probability space $(\Omega, \FFF, \PP)$ with natural filtration $(\fF_t)_{t \geq 0}$. Fix a set of orthonormal basis $\{e_k\}$ of $L^{2}(\RR)$ and a linear operator $\Phi$ on it. The noise $\dot{W}$ in \eqref{eq:main_eq} is the time derivative of the Wiener process $W$, which is given by
\begin{equation*}
W(t,x) := \sum_{k \in \NN} B_{k}(t) \cdot (\Phi e_k)(x). 
\end{equation*}
$W$ is then a Gaussian process on $L^{2}(\RR)$ with covariance operator $\Phi \Phi^{*}$, which can be formally written as $W = \Phi \tilde{W}$ where $\tilde{W}$ is the cylindrical Wiener process. Our assumption on the operator $\Phi$ is the following.

\begin{assumption} \label{as:Phi}
	We assume $\Phi: L^{2}(\RR) \rightarrow \hH$ is a trace-class operator, where $\hH$ is the Hilbert space of real-valued functions with the inner product
	\begin{equation*}
	\scal{f,g}_{\hH} = \sum_{j=0}^{M} \scal{(1+|x|^{K}) f^{(j)}, (1+|x|^{K})g^{(j)}}_{L^{2}}
	\end{equation*}
	for some sufficiently large $K$ and $M$. 
\end{assumption}

Clearly $\hH \hookrightarrow W^{1,p}$ for every $p \in [1,+\infty]$, and our assumption of $\Phi$ implies $\Phi$ is $\gamma$-radonifying from $L^{2}$ to $W^{1,p}$ for every $p \in [1,+\infty]$ ($K, M \geq 10$ would be sufficient for our purpose). A typical example of such an operator is $\Phi e_0 = V(x)$ for some nice function $V$ while $\Phi$ maps all other basis vectors to $0$. In this case, the noise is $B'(t) V(x)$ where $B(t)$ is the standard Brownian motion. 

Note that the assumption for $\hH$ to be such a high regularity space is certainly not strictly necessary. On the other hand, we are not able to treat space-time white noise at this stage, and hence we are not so keen to the exact spatial regularity. 

\begin{rmk}
	With such spacial smoothness of the noise, one may wonder whether the solution theory for \eqref{eq:main_eq} follows directly from the deterministic case. This turns out to be not the case. In fact, the main issue is that the nonlinearity and randomness in the natural solution space together prevent one from setting up a usual fixed point problem. See Sections~\ref{sec:difficulty} and~\ref{sec:strategy} below for more detailed discussions. 
\end{rmk}

Now, we re-write \eqref{eq:main_eq} in its It\^{o} form as
\begin{equation*}
i \partial_t u + \Delta u = |u|^4 u + u \dot{W} - \frac{i}{2} u F_{\Phi}, 
\end{equation*}
where
\begin{equation*}
F_{\Phi}(x) = \sum_{k} (\Phi e_k)^{2}(x)
\end{equation*}
is the It\^{o}-Stratonovich correction. Note that $F_{\Phi}$ is independent of the choice of the basis. The assumption on $\Phi$ guarantees that $\|F_{\Phi}\|_{W^{1,p}_{x}} < \infty$ for all $p \in [1,+\infty]$. 

Before we state the main theorem, we introduce a few notations. Let $\sS(t) = e^{i t \Delta}$ be the linear propagator of the free Schr\"odinger equation. For every interval $\iI$, let
\begin{equation} \label{eq:space}
\xX_{1}(\iI) = L_{t}^{\infty} L_{x}^{2}(\iI) := L^{\infty}(\iI, L^{2}(\RR)), \quad \xX_{2}(\iI) = L_{t}^{5} L_{x}^{10}(\iI) := L^{5}(\iI, L^{10}(\RR)), 
\end{equation}
and $\xX(\iI) = \xX_{1}(\iI) \cap \xX_{2}(\iI)$ in the sense that $\|\cdot\|_{\xX(\iI)} = \|\cdot\|_{\xX_{1}(\iI)} + \|\cdot\|_{\xX_{2}(\iI)}$. We also write $L_{\omega}^{\rho} \xX$ as an abbreviation for $L^{\rho}(\Omega,\xX)$. Our main statement is then the following. 

\begin{thm} \label{th:main}
	For every $u_0 \in L_{\omega}^{\infty}L_{x}^{2}$ and $\fF_0$ measurable, we construct a global flow $u$ adapted to the filtration generated by $W$ such that for every $T>0$ and every $\rho_0 > 5$, we have $u \in L_{\omega}^{\rho_0}\xX(0,T)$, and it satisfies
	\begin{equation} \label{eq:duhamel_main}
	\begin{split}
	u(t) &= \sS(t) u_0 - i \int_{0}^{t} \sS(t-s) \big( |u(s)|^{4} u(s) \big) {\rm d} s\\
	&- i \int_{0}^{t} \sS(t-s) u(s) {\rm d} W_s - \frac{1}{2} \int_{0}^{t} \sS(t-s) \big( F_{\Phi} u(s) \big) {\rm d} s
	\end{split}
	\end{equation}
	in $L_{\omega}^{\rho_0}\xX(0,T)$, where the stochastic integral above is in the It\^{o} sense. Furthermore, there exists $B>0$ depending on $T$, $\rho_0$ and $\|u_0\|_{L_{\omega}^{\infty}L_{x}^{2}}$ only such that
	\begin{equation} \label{eq:main_control}	
	\|u\|_{L_{\omega}^{\rho_{0}}\xX(0,T)} \leq B, 
	\end{equation}
	and we have the pathwise mass conservation in the sense that $\|u(t)\|_{L_x^2} = \|u_0\|_{L_x^2}$ for all $t \in [0,T]$. Moreover, for every $M>0$ and $\delta>0$, there exists $\kappa = \kappa(M,\delta,T,\rho_0)$ such that if $u_0, v_0$ are $\fF_0$ measurable with
	\begin{equation*}
	\|u_0\|_{L_{\omega}^{\infty}L_{x}^{2}} \leq M, \qquad \|v_0\|_{L_{\omega}^{\infty}L_{x}^{2}} \leq M, \qquad \|u_0-v_0\|_{L_{\omega}^{\infty}L_{x}^{2}} \leq \kappa, 
	\end{equation*}
	then the solutions $u$ and $v$ to \eqref{eq:duhamel_main} as constructed in this artible satisfies
	\begin{equation*}
	\|u-v\|_{L_{\omega}^{\rho_0}\xX(0,T)} < \delta. 
	\end{equation*}
	Both $B$ and $\kappa$ above depend on the initial data through their $L_{\omega}^{\infty}L_{x}^{2}$-norm only. 
\end{thm}

\begin{rmk}
	To prove Theorem~\ref{th:main}, we will construct the solution $u$ in the fixed time interval $[0,1]$, and use pathwise mass conservation to extend it to the whole real line. Hence, from now on, we will consider $T=1$ only. 
\end{rmk}

\begin{rmk}
	Our solution $u$ to \eqref{eq:main_eq} is constructed via a fixed approximation procedure rather than a direct contraction principle. Thus, the uniqueness in Theorem~\ref{th:main} is in the quasilinear sense rather than semi-linear sense. More precisely, given every initial data $u_0 \in L_{\omega}^{\infty}L_{x}^{2}$, our construction produces a unique global flow $u$ satisfying \eqref{eq:duhamel_main}. This solution should \textit{not} be confused with the weak solutions which are typically obtained by compactness arguments. 
\end{rmk}

\begin{rmk}
Similar global well-posedness result can also be obtained for the focusing case with the mass of initial data below ground state. More precisely, let $Q$ be the unique positive radial solution to the equation
\begin{equation*}
- \Delta Q + Q = Q^5. 
\end{equation*}
Then as long as $\|u_0\|_{L_{\omega}^{\infty}L_{x}^{2}} < \|Q\|_{L^2}$, there is a unique global flow associated with \eqref{eq:main_eq} with $|u|^4 u$ replaced by $-|u|^4 u$. 
\end{rmk}

\begin{rmk}
In defocusing case, strictly pathwise mass conservation is not necessary.  As far as the noise allows one to have a (uniform) pathwise control for the growth  of mass, similar results will hold. We also expect the same results to hold in dimensions two and three with essentially the same arguments. 
\end{rmk}

\subsection{Background}
The nonlinear Schr\"odinger equation naturally arises from various physics models. The aim of the article is to investigate the impact of a multiplicative noise to the dynamics of mass critical NLS. 

The local well-posedness of the deterministic (defocusing) nonlinear Schr\"odinger equation 
\begin{equation} \label{eq:deterministic}
i\partial_t u+\Delta u=|u|^{p-1} u, \qquad u_0 \in L_{x}^{2}(\RR^d)
\end{equation}
for $p \in [1, 1+\frac{4}{d}]$ is based on Strichartz estimates and has been standard. In short, every $L_{x}^{2}$ initial data gives rise to a space-time function $u$ that satisfies \eqref{eq:deterministic} locally in time. We refer to \cite{cazenave1989some}, \cite{cazenave2003semilinear} and \cite{tao2006nonlinear} for more details. When in the mass subcritical case for general $L_{x}^{2}$ data or mass critical case for small $L^{2}$ data, the local existence time depends on the size of the data only, and one can extend the solution globally in time thanks to the conservation law. The general $L^{2}$ data problem for the mass critical case is much harder, but finally resolved in a series of recent works by Dodson (\cite{dodson2012global, dodson2016global, dodson2467global}). 

We remark that the behaviour of mass-critical problem and mass-subcritical problem are different. When $p-1=\frac{4}{d}$, the linear and nonlinear parts of the equation have the same strength. This will make the problem more subtle.

The study of local well-posedness for mass subcritical (defocusing) stochastic nonlinear Schr\"{o}dinger equation with a conservative multiplicative noise for $L^{2}$ initial data has been initiated in \cite{BD}. Global theory follows from local theory via pathwise mass conservation.  We want to remark here that even the local theory in \cite{BD} depends on the mass conservation law, and in particular is not totally of perturbative nature, which is very different from the deterministic case. There have been subsequent works in various refinements in the stochastic subcritical cases (see for example \cite{Rockner1} and \cite{Rockner2} for extensions to non-conservative cases). The energy subcritical situation has been treated in \cite{BD_H1} (see also \cite{Rockner_scattering}). 

In fact, even in the mass subcritical case (with $|u|^{\frac{4}{d}}$ replaced by $|u|^{\frac{4}{d}-\eps}$ in \eqref{eq:main_eq}), if one tries to directly construct a solution of \eqref{eq:duhamel_main} via contraction map (say in $L_{\omega}^{\rho_0} \xX$ for some $\rho_0$), then one may wonder why \eqref{eq:duhamel_main} is possible to hold since the integrability of the nonlinearity in the probability space can only be in $L^{\rho_{0}/(\frac{4}{d}+1)}$ rather than $L^{\rho_{0}}$. The key point is that, as in the work of \cite{BD}, the construction relies on the pathwise mass conversation law, and is not totally perturbative, so the stochastic process constructed there satisfies \eqref{eq:main_control} for all $\rho_{0}<\rho<\infty$. We also emphasise here that we need the initial data to be bounded in $L^{\infty}_{\omega}L_{x}^{2}$ rather than $L_{\omega}^{\rho}L_{x}^{2}$ for some $\rho$.

In \cite{Hornung}, the author constructed a local solution to the critical equation \eqref{eq:main_eq} stopped at a time when the Strichartz norm of the solution reaches some small positive value. This type of local well-posedness , to the best of our knowledge, cannot be directly combined with the mass conservation law to give a solution global in time, even for small initial data.

\subsection{Obstacles in adapting the deterministic theory}
\label{sec:difficulty}

The classical deterministic theory for the local well-posedness of nonlinear Schr\"{o}dinger equation follows form Picard iteration regime and is of perturbative nature. Let us consider \eqref{eq:deterministic} in $d=1$. To construct a local solution, one needs to show the operator $\Gamma$ given by
\begin{equation}
(\Gamma v)(t):=e^{it\Delta}u_{0}+i\int_{0}^{t}\sS(t-s)\big(|v(s)|^{p-1}v(s) \big) {\rm d} s
\end{equation}
defines a contraction in a suitable function space. Such spaces are indicated by the Strichartz estimates \eqref{eq:Strichartz_1} and \eqref{eq:Strichartz_2} below, for example
\begin{equation} \label{eq:deterministic_space}
\Big\{ v \in \xX(0,T): \|v\|_{\xX_2(0,T)} \leq \eta \Big\}
\end{equation}
for some suitable $T$ and $\eta$, where $\xX$ and $\xX_2$ are as in \eqref{eq:space}. In the subcritical case ($p<5$), whatever $\eta$ is, we can always choose $T$ small enough so that $\Gamma$ forms a contraction in the space \eqref{eq:deterministic_space}. This is because the Strichartz estimate would give us the factor $T^{1-\frac{p-1}{4}}$ in front of the nonlinearity (see Proposition~\ref{pr:pre_nonlinear} below). This factor is not available in the critical case when $p=5$. Nevertheless, one can still choose $\eta$ and $T$ small enough depending on the initial data so that $\Gamma$ still defines a contraction in \eqref{eq:deterministic_space}. But the key point is that the existence time $T$ depends on the profile of the initial data, so one cannot easily extend it to any fixed time. 

We now turn to the stochastic problem
\begin{equation*}
i \partial_t u + \Delta u = |u|^{p-1} u + u \circ \dot{W}
\end{equation*}
with $L^2$ initial data $u_0$. The natural space to search for the solution is $L_{\omega}^{\rho} \xX(0,T)$ for some $T>0$ and $\rho \geq 1$. However, the first obstacle is that the analogous Duhamel operator in this case does not even map $L_{\omega}^{\rho}\xX(0,T)$ to itself, whatever $\rho$ and $T$ are! This is because if $v \in L_{\omega}^{\rho}$, then the nonlinearity can only be in $L_{\omega}^{\rho/p}$. 

To overcome this problem, \cite{BD} introduced a truncation to kill the nonlinearity whenever $\|u\|_{\xX_{2}(0,t)}$ reaches $m$. In the fixed point problem, one then can replace $p$ powers of $u$ by $m^p$, thus yielding an operator mapping a space to itself (with the size of the norms depending on $m$). In the subcritical case $p<5$, one can make the time $T = T_m$ small enough to get a contraction, so that one gets a local solution in $u_{m} \in L_{\omega}^{\rho}(0,T_m)$. It can be extended globally due to pathwise mass conservation. Finally, they showed that the sequence of solutions $\{u_{m}\}$ actually converges as $m \rightarrow +\infty$. This relies on a uniform bound on $\{u_{m}\}$ when $p<5$. 

However, this construction does not extend to the critical case $p=5$, as there is no positive power of $T$ available to compensate the largeness of $m$ (unless $m$ itself is very small, in which case does not relate to the original problem any more).

\subsection{Overview of construction}
\label{sec:strategy}

From now on, we fix arbitrary $M > 0$ and $u_0$ independent of $W$ such that $\|u_0\|_{L_{\omega}^{\infty}L_{x}^{2}} \leq M$. Our aim is to construct a process $u \in L_{\omega}^{\rho_0}\xX(0,1)$ satisfying \eqref{eq:duhamel_main}. 

The starting point of our construction is the existence of the solution to the truncated subcritical problem from \cite{BD}. Then, we show that for any truncation, the subcritical solutions converge to the solution of the corresponding truncated critical problem. Finally, we obtain uniform bounds on the solutions to the truncated critical equations that allow us to remove the truncation. 

To be precise, we let $\theta: \RR \rightarrow \RR^{+}$ be a smooth function with compact support in $(-2,2)$ and $\theta = 1$ on $[-1,1]$. For every $m>0$, let $\theta_{m}(x) = \theta(x/m)$. Consider the truncated sub-critical equation
\begin{equation} \label{eq:trun_sub}
i \partial_t u_{m,\eps}+\Delta u_{m,\eps} = \theta_{m} \big(\|u_{m,\eps}\|_{\xX_{2}(0,t)}^{5} \big) \nN^{\eps}(\uem) + u_{m,\eps} \circ \dot{W}, \quad u_0 \in L_{\omega}^{\infty}L_{x}^{2}, 
\end{equation}
where $\nN^{\eps}(u) = |u|^{4-\eps}u$, and $u_0$ is $\fF_0$ measurable. We also used $\eps$ to denote $4-p$. The following global existence theorem is contained in \cite{BD}. 

\begin{thm} [De Bouard--Debussche]
	\label{th:global_trun}
	Let $\|u_0\|_{L_{\omega}^{\infty}L_{x}^{2}} < +\infty$. For every $m,\eps>0$ and every sufficiently large $\rho$, there exists $T = T(m, \eps, \rho, \|u_0\|_{L_{\omega}^{\infty}L_{x}^{2}})$ such that the equation \eqref{eq:trun_sub} has a unique solution $\uem$ in $L_{\omega}^{\rho}\xX(0,T)$. It satisfies the Duhamel formula
	\begin{equation} \label{eq:duhamel_trun_sub}
	\begin{split}
	u_{m,\eps}(t) &= \sS(t) u_0 - i \int_{0}^{t} \sS(t-s) \Big( \theta_{m} \big(\|\uem\|_{\xX_{2}(0,s)}^{5} \big) \nN^{\eps}(\uem(s))\Big) {\rm d} s\\
	&- i \int_{0}^{t} \sS(t-s) \uem(s) {\rm d} W_s - \frac{1}{2} \int_{0}^{t} \sS(t-s) \big(F_{\Phi} \uem(s) \big) {\rm d} s
	\end{split}
	\end{equation}
	in $L_{\omega}^{\rho} \xX(0,T)$. Furthermore, we have pathwise mass conservation in the sense that  
	\begin{equation}
	\|u_{m,\eps}(t)\|_{L_{x}^{2}} = \|u_0\|_{L_{x}^{2}}, \qquad \forall t \in [0,T]
	\end{equation}
	almost surely. As a consequence, one can iterate the construction to get a global in time flow $\uem$.
\end{thm}

\begin{remark}
	Strictly speaking, due to the truncation in the nonlinearity, the equation \eqref{eq:trun_sub} is \textit{not} translation invariant in time. Thus, it is not immediately obvious that a local theory with mass conservation can imply global existence. However, the nonlinearity satisfies a bound of the form
	\begin{equation*}
	\Big\|\int_{0}^{t}\theta_{m} \big(\|\uem\|_{\xX_{2}(0,s)}^{5} \big) \nN^{\eps}\big(\uem(s)\big) {\rm d}s \Big\|_{\xX(0,T)} \leq C T^{\frac{\eps}{4}} m^{4}\|\uem\|_{\xX(0,T)}. 
	\end{equation*}
	The right hand side above is translation invariant in time, and hence one can iterate the local construction to get a global flow.
	
	Also, we put $\theta_{m}\big( \|u\|_{\xX_2(0,s)}^{5} \big)$ instead of $\theta_{m} \big( \|u\|_{\xX_2(0,s)} \big)$ to make the argument of $\theta_m$ additive in time. This will simplify a few arguments in uniform-in-$m$ bounds in Section~\ref{sec:uniform_m} below. 
\end{remark}

Theorem \ref{th:global_trun} follows from standard Picard iteration (which is purely perturbative) and relies essentially on Strichartz estimates and Burkholder inequality. We refer to \cite[Proposition 3.1]{BD} for more details of the proof. The key part of \cite{BD}, however, is a uniform in $m$ bound for the sequence $\{u_{m,\eps}\}_{m}$ for every fixed $\eps>0$. This allows the authors to show the convergence of $u_{m,\eps}$ to a limit $u_{\infty,\eps}$, and that the limit solves the corresponding subcritical equation without truncation. 

On the other hand, in order to obtain a solution for the critical problem \eqref{eq:main_eq}, we need to send $\eps \rightarrow 0$ and $m \rightarrow +\infty$. Instead of starting from the subcritical solution $u_{\infty,\eps}$ directly, we fix $m$ and send $\eps \rightarrow 0$ first. 

\begin{prop}\label{pr:converge_e}
	Let $\uem$ be the solution to \eqref{eq:trun_sub} with initial data $u_0 \in L_{\omega}^{\infty}L_{x}^{2}$, so that it satisfies the Duhamel's formula \eqref{eq:duhamel_trun_sub}. Then, for every $\rho_0 \geq 5$, the sequence $\{u_{m,\eps}\}$ is Cauchy in $L_{\omega}^{\rho_0}(0,1)$. Furthermore, the limit $u_m$ satisfies the Duhamels' formula
	\begin{equation} \label{eq:trun_critical}
	\begin{split}
	u_{m}(t) &= \sS(t) u_0 - i \int_{0}^{t} \sS(t-s) \Big( \theta_{m}\big(\|u_{m}\|_{\xX_{2}(0,t)}^{5}\big) \nN\big(u_{m}(s)\big)\Big) {\rm d}s\\
	&- i \int_{0}^{t} \sS(t-s) u_{m}(s) {\rm d} W_s - \frac{1}{2} \int_{0}^{t} \sS(t-s) \big( F_{\Phi} u_{m}(s) \big) {\rm d}s
	\end{split}
	\end{equation}
	in $L_{\omega}^{\rho_0}\xX(0,1)$, where $\nN(v) = |v|^{4}v$. 
\end{prop}

Our next step is to show that the sequence $\{u_{m}\}$ also converges to a limit in $L_{\omega}^{\rho_0} \xX(0,1)$. The main ingredient is the following uniform bound. 

\begin{prop} \label{pr:uniform_m}
	Let $u_m$ be the process satisfying \eqref{eq:trun_critical} with $\|u_0\|_{L_{\omega}^{\infty}L_{x}^{2}} \leq M$. Then for every $\rho>0$, there exists $B = B(M, \rho)$ such that
	\begin{equation*}
	\|u_m\|_{L_{\omega}^{\rho}\xX(0,1)} \leq B
	\end{equation*}
	for all $m$. 
\end{prop}

In order to show the convergence of $\{u_{m}\}$ in $L_{\omega}^{\rho_0}\xX(0,1)$ as $m \rightarrow +\infty$, it is essential that the uniform bound above holds with a strictly larger $\rho$ (and we need $\rho > 5 \rho_0$ in our case). With the help of Proposition~\ref{pr:uniform_m}, we can prove that $\{u_{m}\}$ is Cauchy in $L_{\omega}^{\rho_0}\xX(0,1)$, and the limit $u$ satisfies the Duhamel formula \eqref{eq:duhamel_main} in the same space and is stable under perturbation of initial data. This concludes the proof of Theorem~\ref{th:main}.

\begin{rmk}
	As mentioned earlier, the key point in \cite{BD} is to show $u_{m,\eps} \rightarrow u_{\infty,\eps}$ for every positive $\eps$, and that the limit satisfies the subcritical equation with nonlinearity $\nN^\eps$. Hence, it is also natural to start from $u_{\infty,\eps}$ and then sending $\eps \rightarrow 0$ to construct the candidate solution. Such construction requires a uniform-in-$\eps$ bound without the presence of truncation $m$. The proof of these bounds and subcritical approximations is more technically involved and uses concentration compactness. They will be treated in the separate article \cite{Subcritical_approx}. 
\end{rmk}

\begin{rmk}
	We finally remark that it might be possible to use rough path theory and regularity structures developed in \cite{rp, controlled_rp, rs} to develop a pathwise solution theory to \eqref{eq:main_eq}. This would avoid the problem of the stochastic integrability, and may give a direct construction via a fixed point argument. We plan to investigate this issue in future work.
\end{rmk}

\subsection*{Structure of the article}

The rest of the article is organized as follows. We first give some preliminary lemmas and bounds on the equation in Section~\ref{sec:preliminary}. These bounds will be used throughout the article. Section~\ref{sec:converge_e}  gives the convergence in of subcritical solutions to the critical one with the presence of truncation.  In Section~\ref{sec:uniform_m}, we prove uniform bounds on the family of solutions to the truncated critical equations. Finally, in Section~\ref{sec:converge_m}, these bounds are used to show that the truncation can be removed, thus yielding a construction of the solution to \eqref{eq:main_eq} as well as its stability.

\subsection*{Notations}

We now introduce the notations used in this article. For any interval $\iI$, we use $L_{t}^{q} L_{x}^{r} (a,b)$ to denote the space $L^{q}(\iI, L^{r}(\RR))$, and we also write $L_{\omega}^{\rho} \yY = L^{\rho}(\Omega, \yY)$. We fix the spaces $\xX_1$ and $\xX_2$ to be
\begin{equation*}
\xX_1 (\iI) = L_{t}^{\infty} L_{x}^{2}(\iI), \qquad \xX_{2}(\iI) = L_{t}^{5} L_{x}^{10}(\iI), 
\end{equation*}
and $\xX(\iI) = \xX_{1}(\iI) \cap \xX_{2}(\iI)$ with the norm being their sum. Some intermediate steps in the proof require us to go to a higher regularity space than $L_{x}^{2}$, so we let $\xX^{1}(\iI)$ be the space of functions such that
\begin{equation*}
\|u\|_{\xX^{1}(\iI)} \eqdef \|u\|_{\xX(\iI)} + \|\d_x u\|_{\xX(\iI)} < +\infty. 
\end{equation*}
Throughout this article, we fix an arbitrary $\rho_0 > 5$, and all the dependence of $\rho_0$ will be omitted below for simplicity. 

We also write $\nN(u) = |u|^{4} u$ and $\nN^{\eps}(u) = |u|^{4-\eps} u$. Finally, we fix $\theta$ to be a non-negative smooth function on $\RR$ with compact support in $(-2,2)$ such that $\theta(x) = 1$ for $|x| \leq 1$. For every $m>0$, we let $\theta_{m}(x) = \theta(x/m)$. 

Finally, $C, C_{m,M}, C_{\rho}$ etc. denote constants whose value may change from line to line. The dependence of these constants on certain parameters are indicated by the subscripts. Also, since we fix $\rho_0 \geq 5$ throughout the article, we omit the dependence on $\rho_0$ in all the bounds below.

\subsection*{Acknowledgements}

We thank Carlos Kenig and Gigliola Staffilani for helpful discussions and comments. CF would also like to thank Chuntian Wang and Jie Zhong for discussions. This work was initiated during the Fall 2015 program ``New Challenges in PDE: Deterministic Dynamics and Randomness in High and Infinite Dimensional Systems" held at MSRI, Berkeley. We thank MSRI for providing a stimulating mathematical environment. 

Part of this work was conducted when CF was a graduate student in MIT. CF was partially supported by NSF DMS 1362509 and DMS 1462401. WX acknowledges the support from the Engineering and Physical Sciences Research Council through the fellowship EP/N021568/1. 

The material presented in this article is contained in the following two preprints of the authors: arXiv:1803.03257 and arXiv:1807.04402. After we finished those two preprints, we decided to re-organise the material to form the current paper. The other part of arXiv1807.04402, which has not been covered in the current article, will be re-written as another independent one {\cite{Subcritical_approx}}. Only the current article and \cite{Subcritical_approx} will be submitted for journal publication. 

\section{Preliminaries}\label{sec:preliminary}

\subsection{The Wiener process and Burkholder inequality}

Our main assumption on the noise $W$ is that it can be written as $W = \Phi \tilde{W}$, where $\tilde{W}$ is the cylindrical Wiener process on $L^2(\RR)$, and $\Phi$ is a trace-class operator satisfying Assumption~\ref{as:Phi}. 

We now introduce the notion of $\gamma$-radonifying operators since the Burkholder inequality we use below is most conveniently expressed with this notion. A linear operator $\Gamma: \bB \rightarrow \tilde{\hH}$ from a Banach space $\bB$ to a Hilbert space $\tilde{\hH}$ is $\gamma$-radonifying if for any sequence $\{\gamma_k\}_k$ of independent standard normal random variables on a probability space $(\tilde{\Omega}, \tilde{\fF}, \tilde{\PP})$ and any orthonormal basis $\{e_k\}_k$ of $\tilde{\hH}$, the series $\sum_k \gamma_k \Gamma e_k$ converges in $L^2(\tilde{\Omega}, \bB)$. The $\gamma$-radonifying norm of the operator $\Gamma$ is then defined by
\begin{equation*}
\|\Gamma\|_{\R(\tilde{H}, \bB)} = \Big( \EE \|\sum_k \gamma_k \Gamma e_k\|_{\bB}^{2} \Big)^{\frac{1}{2}}, 
\end{equation*}
which is independent of the choice of $\{\gamma_k\}$ or $\{e_k\}$. We then have
\begin{equation*}
\|\Phi\|_{\R(L_{x}^{2}, \hH)} < +\infty
\end{equation*}
for the Hilbert space $\hH$ specified in Assumption~\ref{as:Phi}, where the Hilbert space $\tilde{H}$ here is $L_x^2(\RR)$. We also need the following factorisation lemma. 

\begin{lem} \label{le:factorisation}
	Let $\kK$ be a Hilbert space, and $\eE$ and $\bB$ be Banach spaces. For every $\Gamma \in \R(\kK,\eE)$ and $\tT \in \LLL(\eE,\bB)$, we have $\tT \circ \Gamma \in \R(\kK,\bB)$ with the bound
	\begin{equation*}
	\|\tT \circ \Gamma\|_{\R(\kK,\bB)} \leq \|\tT\|_{\LLL(\eE,\bB)} \|\Gamma\|_{\R(\kK,\eE)}. 
	\end{equation*}
	In particular, if $\eE = L^p$ and $\tT$ is given by the multiplication of an $L^q$ function $\sigma$ with $\frac{1}{p} + \frac{1}{q} = \frac{1}{r} \leq 1$, then $\sigma \Gamma \in \R(\kK,L^r)$ with
	\begin{equation*}
	\|\sigma \Gamma\|_{\R(\kK,L^r)} \leq \|\sigma\|_{L^q} \|\Gamma\|_{\R(\kK,L^p)}. 
	\end{equation*}
\end{lem}
\begin{proof}
	The first claim is same as \cite[Lemma 2.1]{BD}. The second claim is an immediate consequence of the first one and H\"{o}lder's inequality. 
\end{proof}

The Burkholder inequality (\cite{BDG, Burkholder}) is very useful in controlling moments of the supremum of a martingale. We will make use of the following version. 

\begin{prop} \label{pr:Burkholder}
	Let $W = \Phi \tilde{W}$ with $\tilde{W}$ and $\Phi$ be as described above. Let $\sigma$ be adapted to $\fF_t$. Then, for every $p \in [2,\infty)$, every $\rho \in [1,\infty)$ and every interval $[a,b]$, we have
	\begin{equation*}
	\EE \sup_{t \in [a,b]} \Big\| \int_{a}^{t} \sigma(s) {\rm d} W_s \Big\|_{L^p}^{\rho} \leq C \EE \Big( \int_{a}^{b} \|\sigma(s) \Phi\|_{\R(L^{2}, L^{p})}^{2} {\rm d} s \Big)^{\frac{\rho}{2}}. 
	\end{equation*}
	The constant $C$ depends on $p$ and $\rho$ only. 
\end{prop}

The proof can be found, for example, in \cite[Theorem 2.1]{BP}. More details about this version of the inequality can be found in \cite{Brzezniak, UMD}.

\subsection{Dispersive and Strichartz estimates}

We give some dispersive and Strichartz estimates of the free Schr\"odinger operator, which will be used throughout the article. They are now standard and can be found in \cite{cazenave2003semilinear}, \cite{keel1998endpoint} and \cite{tao2006nonlinear}. 

Recall that $\sS(t) = e^{i t \Delta}$. We need the following dispersive estimates of the semigroup $\sS$ in $d=1$. 

\begin{prop} [Dispersive estimates]
	\label{pr:dispersive}
	There exists a universal constant $C>0$ such that
	\begin{equation} \label{eq:dispersive}
	\|\sS(t) f\|_{L^{p'}} \leq C t^{\frac{1}{p}-\frac{1}{2}} \|f\|_{L^{p}}
	\end{equation}
	for every $p \in [1,2]$ and every $f \in L^{p}(\RR)$. Here, $p'$ is the conjugate of $p$. 
\end{prop}

We now turn to Strichartz estimates. A pair of real numbers $(q,r)$ is called an admissible pair (for $d=1$) if
\begin{equation} \label{eq:admissible}
\frac{2}{q} + \frac{1}{r} = \frac{1}{2}. 
\end{equation}
The following Strichartz estimates give the right space to build solutions. 

\begin{prop} [Strichartz estimates]
	\label{pr:Strichartz}
	For every two admissible pairs $(q,r)$ and $(\tilde{q}, \tilde{r})$, there exists $C>0$ such that
	\begin{equation} \label{eq:Strichartz_1}
	\|\sS(t) f\|_{L_{t}^{q}L_{x}^{r}(\RR)} \leq C \|f\|_{L_{x}^{2}}
	\end{equation}
	for all $f \in L_x^2$, and
	\begin{equation} \label{eq:Strichartz_2}
	\Big\|\int_{a}^{t} \sS(t-s) \sigma(s) {\rm d} s \Big\|_{L_{t}^{q}L_{x}^{r}(\iI)} \leq C \|\sigma\|_{L_{t}^{\tilde{q}'}L_{x}^{\tilde{r}'}(\iI)}
	\end{equation}
	for all $\sigma \in L_{t}^{\tilde{q}'}L_{x}^{\tilde{r}'}(\iI)$. Here, $\tilde{q}'$, $\tilde{r}'$ are conjugates of $\tilde{q}$ and $\tilde{r}$. The proportionality constants are independent of $f$, $\sigma$ and the length of the time interval. 
\end{prop}

\subsection{Some bounds on the equation}
\label{sec:bounds}

We now give some bounds arising from various parts of the equation \eqref{eq:duhamel_trun_sub}. In what follows, $\sigma$ and $\tilde{\sigma}$ denote processes satisfying different assumptions in various statements. But we should think of them as the solution $\uem$ to \eqref{eq:trun_sub}, or the difference between its two solutions starting from different initial data. Also recall the notations $\xX_{1}(\iI) = L_{t}^{\infty}L_{x}^{2}(\iI)$, $\xX_{2}(\iI) = L_{t}^{5} L_{x}^{10}(\iI)$, and $\xX = \xX_1 \cap \xX_2$. We will always use $\iI$ to denote the interval $[a,b]$ concerned in the contexts below. 

\begin{prop} \label{pr:pre_nonlinear}
	There exists a universal constant $C>0$ such that
	\begin{equation} \label{eq:pre_nonlinear_eps}
	\Big\| \int_{a}^{t} \sS(t-s) \big( \tilde{\sigma}(s) |\sigma(s)|^{4-\eps} \big) {\rm d}s \Big\|_{\xX(\iI)} \leq C (b-a)^{\frac{\eps}{4}} \|\tilde{\sigma}\|_{\xX_{1}(\iI)} \|\sigma\|_{\xX_{2}(\iI)}^{4-\eps}
	\end{equation}
	for all $a < b$ and all $\eps \in [0,1]$. In the case $\eps=0$, we have
	\begin{equation} \label{eq:pre_nonlinear_critical}
	\Big\| \int_{a}^{t} \sS(t-s) \big( \tilde{\sigma}(s) |\sigma(s)|^{4} \big) {\rm d} s\Big\|_{\xX(\iI)} \leq C \|\tilde{\sigma}\|_{\xX_2(\iI)} \|\sigma\|_{\xX_2(\iI)}^{4}, 
	\end{equation}
	where both terms on the right hand side are with $\xX_2$-norms, and the bound is uniform over all $a<b$. 
\end{prop}
\begin{proof}
	The pair $(q_\eps', r_\eps') = (\frac{20}{16+\eps}, \frac{10}{9-\eps})$ is dual of the Strichartz pair $(\frac{20}{4-\eps}, \frac{10}{1+\eps})$. So by Strichartz estimate \eqref{eq:Strichartz_2}, we have
	\begin{equation*}
	\Big\| \int_{a}^{t} \sS(t-s) \big( \tilde{\sigma}(s) |\sigma(s)|^{4-\eps} \big) {\rm d}s \Big\|_{\xX(\iI)} \leq C \big\|\; \tilde{\sigma} |\sigma|^{4-\eps}\; \big\|_{L_{t}^{q_{\eps}'}L_{x}^{r_\eps'}(\iI)},  
	\end{equation*}
	Note that this $C$ can be taken independent of $\eps \in [0,1]$. Repeated applications of H\"{o}lder give
	\begin{equation*}
	\big\|\; \tilde{\sigma} |\sigma|^{4-\eps}\; \big\|_{L_{t}^{q_{\eps}'}L_{x}^{r_\eps'}(\iI)} \leq (b-a)^{\frac{\eps}{4}} \|\tilde{\sigma}\|_{\xX_1(\iI)} \|\sigma\|_{\xX_2(\iI)}^{4-\eps}. 
	\end{equation*}
	This proves \eqref{eq:pre_nonlinear_eps}. As for \eqref{eq:pre_nonlinear_critical}, we use Strichartz estimate \eqref{eq:Strichartz_2} with $(\tilde{q}', \tilde{r}') = (1,2)$ to get the bound
	\begin{equation*}
	\Big\| \int_{a}^{t} \sS(t-s) \big( \tilde{\sigma}(s) |\sigma(s)|^{4} \big) {\rm d} s\Big\|_{\xX(\iI)} \leq C \|\tilde{\sigma} \sigma^4 \|_{L_{t}^{1}L_{x}^{2}(\iI)}. 
	\end{equation*}
	The claim then follows from H\"older inequality. 
\end{proof}

In the subcritical situation $\eps>0$, the factor $(b-a)^{\frac{\eps}{4}}$ is crucial for constructing the local solution via contraction. However, to show the convergence of solutions as $\eps \rightarrow 0$, we need to get uniform in $\eps$ estimates. Hence, this factor will be of little use to us. In what follows, we will always use the following bound. 

\begin{cor} \label{cor:pre_nonlinear}
	There exists $C>0$ such that
	\begin{equation*}
	\Big\| \int_{a}^{t} \sS(t-s) \big( \tilde{\sigma}(s) |\sigma(s)|^{4-\eps} \big) {\rm d}s \Big\|_{\xX(\iI)} \leq C \|\tilde{\sigma}\|_{\xX_{1}(\iI)} \|\sigma\|_{\xX_{2}(\iI)}^{4-\eps}
	\end{equation*}
	for all $\iI = [a,b] \subset [0,1]$ and every $\eps \in [0,1]$. 
\end{cor}

Let $\sigma$ be a process adapted to the filtration $\fF_{t}$. Let
\begin{equation} \label{eq:martingales}
\begin{split}
\mM_{1,a}^{*}(t) &= \sup_{a \leq r_{1} \leq r_{2} \leq t} \Big\| \int_{r_1}^{r_2} \sS(t-s) \sigma(s) {\rm d} W_s \Big\|_{L_{x}^{2}},\\
\mM_{2,a}^{*}(t) &= \sup_{a \leq r_{1} \leq r_{2} \leq t} \Big\| \int_{r_1}^{r_2} \sS(t-s) \sigma(s) {\rm d} W_s \Big\|_{L_{x}^{10}}, 
\end{split}
\end{equation}
where both integrals are in the It\^o sense. We have the following proposition. 

\begin{prop} \label{pr:pre_stochastic}
	For every $\rho \geq 5$, we have the bounds
	\begin{equation*}
	\begin{split}
	\|\mM_{1,a}^{*}\|_{L_{\omega}^{\rho}L_{t}^{\infty}(\iI)} &\leq C_{\rho} (b-a)^{\frac{1}{2}} \|\Phi\|_{\R(L_{x}^{2}, L_{x}^{\infty})} \|\sigma\|_{L_{\omega}^{\rho} \xX_{1}(\iI)};\\
	\|\mM_{2,a}^{*}\|_{L_{\omega}^{\rho}L_{t}^{5}(\iI)} &\leq C_{\rho} (b-a)^{\frac{3}{10}} \|\Phi\|_{\R(L_{x}^{2}, L_{x}^{5/2})} \|\sigma\|_{L_{\omega}^{\rho} \xX_{1}(\iI)}, 
	\end{split}
	\end{equation*}
	where the proportionality constants depend on $\rho$ only. 
\end{prop}
\begin{proof}
	We first treat $\mM_{1,a}^{*}$. By semigroup and unitary properties of $\sS$, we have
	\begin{equation*}
	\mM_{1,a}^{*}(t) = \sup_{a \leq r_{1} \leq r_{2} \leq t} \Big\| \int_{r_1}^{r_2} \sS(-s) \sigma(s) {\rm d} W_{s} \Big\|_{L_{x}^{2}} \leq 2 \sup_{a \leq r \leq b} \Big\| \int_{a}^{r} \sS(-s) \sigma(s) {\rm d} W_{s} \Big\|_{L_{x}^{2}}. 
	\end{equation*}
	The right hand side above does not depend on $t$, so we have
	\begin{equation*}
	\|\mM_{1,a}^{*}\|_{L_{t}^{\infty}(\iI)} \leq 2 \sup_{a \leq r \leq b} \Big\| \int_{a}^{r} \sS(-s) \sigma(s) {\rm d} W_{s} \Big\|_{L_{x}^{2}}. 
	\end{equation*}
	Now, since the process $s \mapsto \sS(-s) \sigma(s)$ is adapted to $\FFF_t$, we can apply Burkholder inequality in Proposition~\ref{pr:Burkholder} (with $\sS(-s) \sigma(s)$ replacing $\sigma(s)$) to get
	\begin{equation*}
	\EE \|\mM_{1,a}^{*}\|_{L_{t}^{\infty}(\iI)}^{\rho} \leq C_{\rho} \EE \Big( \int_{a}^{b} \|\sS(-s) \sigma(s) \Phi\|_{\R(L_{x}^{2}, L_{x}^{2})}^{2} {\rm d}s \Big)^{\frac{\rho}{2}}. 
	\end{equation*}
	Using again the unitary property of $\sS(-s)$ and Lemma \ref{le:factorisation}, we have
	\begin{equation*}
	\|\sS(-s) \sigma(s) \Phi\|_{\R(L_{x}^{2}, L_{x}^{2})} \leq \|\sS(-s) \sigma(s)\|_{\LLL(L_{x}^{\infty}, L_{x}^{2})} \|\Phi\|_{\R(L_{x}^{2}, L_{x}^{\infty})} \leq \|\sigma(s)\|_{L_{x}^{2}} \|\Phi\|_{\R(L_{x}^{2}, L_{x}^{\infty})}. 
	\end{equation*}
	Plugging it back to the above bound for $\EE \|M_{1,a}^{*}\|_{L_{t}^{\infty}}^{\rho}$ and applying H\"{o}lder, we get
	\begin{equation*}
	\EE \|\mM_{1,a}^{*}\|_{L_{t}^{\infty}(\iI)}^{\rho} \leq C_{\rho} (b-a)^{\frac{\rho}{2}} \|\Phi\|_{\R(L_{x}^{2}, L_{x}^{\infty})}^{\rho} \EE \|\sigma\|_{\xX_{1}(\iI)}^{\rho}. 
	\end{equation*}
	Taking $\rho$-th root on both sides gives the desired bound for $\mM_{1,a}^{*}$. As for $\mM_{2,a}^{*}$, since $\rho \geq 5$, we can use Minkowski to change the order of integration and then apply H\"{o}lder so that
	\begin{equation} \label{eq:minkowski_exchange}
	\|\mM_{2,a}^{*}\|_{L_{\omega}^{\rho} L_{t}^{5}(\iI)} \leq \|\mM_{2,a}^{*}\|_{L_{t}^{5}(\iI, L_{\omega}^{\rho})} \leq (b-a)^{\frac{1}{5}} \sup_{t \in [a,b]} \|\mM_{2,a}^{*}(t)\|_{L_{\omega}^{\rho}}. 
	\end{equation}
	Since
	\begin{equation*}
	\mM_{2,a}^{*}(t) \leq 2 \sup_{0 \leq \tau \leq t} \Big\| \int_{a}^{\tau} \sS(t-s) \sigma(s) {\rm d} W_{s} \Big\|_{L_{x}^{10}}, 
	\end{equation*}
	we use Burkholder inequality to get
	\begin{equation*}
	\|\mM_{2,a}^{*}(t)\|_{L_{\omega}^{\rho}}^{\rho} \leq C_{\rho} \EE \Big( \int_{a}^{t} \|\sS(t-s) \sigma(s) \Phi\|_{\R(L_{x}^{2}, L_{x}^{10})}^{2} {\rm d}s \Big)^{\frac{\rho}{2}}. 
	\end{equation*}
	Now, applying the dispersive estimate \eqref{eq:dispersive} and Lemma \ref{le:factorisation} to the integrand, we get
	\begin{equation*}
	\begin{split}
	\|\sS(t-s) \sigma(s) \Phi\|_{\R(L_{x}^{2}, L_{x}^{10})} &\leq \|\sS(t-s) \sigma(s)\|_{\LLL(L_{x}^{5/2}, L_{x}^{10})} \|\Phi\|_{\R(L_{x}^{2}, L_{x}^{5/2})}\\ 
	&\leq C (t-s)^{-\frac{2}{5}} \|\sigma(s)\|_{L_{x}^{2}} \|\Phi\|_{\R(L_{x}^{2}, L_{x}^{5/2})}. 
	\end{split}
	\end{equation*}
	Substituting it back into the bound for $\|\mM_{2,a}^{*}(t)\|_{L_{\omega}^{\rho}}^{\rho}$, we get
	\begin{equation*}
	\|\mM_{2,a}^{*}(t)\|_{L_{\omega}^{\rho}}^{\rho} \leq C_{\rho} (b-a)^{\frac{\rho}{10}} \|\Phi\|_{\R(L_{x}^{2}, L_{x}^{5/2})}^{\rho} \EE \|\sigma\|_{\xX_{1}(\iI)}^{\rho}. 
	\end{equation*}
	Note that the right hand side above does not depend on $t$. So taking the $\rho$-th root on both sides and then supremum over $t \in [a,b]$, and combining it with \eqref{eq:minkowski_exchange}, we obtain the desired control for $\mM_{2,a}^{*}$. 
\end{proof}

\begin{rmk}
	The exact value of the exponent of $(b-a)$ is not very important. However, it is crucial that the exponent is strictly positive. This will enable us to absorb the term on the right hand side of the equation to the left hand side with a short time period. The same is true for the bounds in Proposition~\ref{pr:pre_correction} below. 
\end{rmk}

\begin{rmk}
	For both of the bounds in Proposition~\ref{pr:pre_stochastic}, one can slightly modify the argument to control the left hand side by the $L_{\omega}^{\rho} \xX_{2}(\iI)$ norm. We choose the $L^2$ based norm $\xX_1$ for convenience of use later. Also Assumption~\ref{as:Phi} ensures that all the $\gamma$-radonifying norms of $\Phi$ appearing above are finite. 
\end{rmk}

We finally give bounds on the correction term. 

\begin{prop} \label{pr:pre_correction}
	There exists $C>0$ such that
	\begin{equation*}
	\begin{split}
	\Big\|\int_{a}^{t} \sS(t-s) \big( F_{\Phi} \sigma(s) \big) {\rm d}s \Big\|_{\xX_{1}(\iI)} &\leq C (b-a) \|F_{\Phi}\|_{L_{x}^{\infty}} \|\sigma\|_{\xX_{1}(\iI)};\\
	\Big\|\int_{a}^{t} \sS(t-s) \big( F_{\Phi} \sigma(s) \big) {\rm d}s \Big\|_{\xX_{2}(\iI)} &\leq C (b-a)^{\frac{4}{5}} \|F_{\Phi}\|_{L_{x}^{5/2}} \|\sigma\|_{\xX_{1}(\iI)}
	\end{split}
	\end{equation*}
	for all $\sigma \in \xX_1(\iI)$. 
\end{prop}
\begin{proof}
	We first look at $\xX_{1}$-norm of $A$. By unitary property of $\sS$ and H\"{o}lder, we have
	\begin{equation*}
	\Big\|\int_{a}^{t} \sS(t-s) \big( F_{\Phi} \sigma(s) \big) {\rm d}s \Big\|_{L_{x}^{2}} \leq \int_{a}^{t} \|F_{\Phi} \sigma(s)\|_{L_{x}^{2}} {\rm d}s \leq (b-a) \|F_{\Phi}\|_{L_{x}^{\infty}} \|\sigma\|_{\xX_1(\iI)}. 
	\end{equation*}
	The right hand side above does not depend on $t$, so we have proved the first bound. As for the one involving the $\xX_{2}$-norm, by dispersive estimate \eqref{eq:dispersive} and H\"{o}lder, we have
	\begin{equation*}
	\begin{split}
	\Big\|\int_{a}^{t} \sS(t-s) \big( F_{\Phi} \sigma(s) \big) {\rm d}s\Big\|_{L_{x}^{10}} &\leq C \int_{a}^{t} (t-s)^{-\frac{2}{5}} \|F_{\Phi} \sigma(s)\|_{L_{x}^{10/9}} {\rm d}s\\
	&\leq C (b-a)^{\frac{3}{5}} \|F_{\Phi}\|_{L_{x}^{5/2}} \|\sigma\|_{\xX_1(\iI)}.
	\end{split}
	\end{equation*}
	where we have integrated $s$ out and replaced $t-a$ by $b-a$ as an upper bound. Note that the bound above does not depend on $t$. Taking $L_{t}^{5}(\iI)$-norm then immediately gives
	\begin{equation*}
	\Big\|\int_{a}^{t} \sS(t-s) \big( F_{\Phi} \sigma(s) \big) {\rm d}s\Big\|_{\xX_{2}(\iI)} \leq C (b-a)^{\frac{4}{5}} \|F_{\Phi}\|_{L_{x}^{5/2}} \|\sigma\|_{\xX_{1}(\iI)}. 
	\end{equation*}
	This completes the proof of the proposition. 
\end{proof}

\section{Convergence in $\eps$ -- proof of Proposition~\ref{pr:converge_e}}
\label{sec:converge_e}

The aim of this section is to prove Proposition~\ref{pr:converge_e}. We will show that for every $m$, the sequence of solutions $\{u_{m,\eps}\}_{\eps>0}$ is Cauchy in $L_{\omega}^{\rho_0} \xX(0,1)$, and that the limit $u_{m}$ satisfies the corresponding Duhamel's formula. All estimates below are uniform in $\eps$ but depend on $\rho_0$, $m$ and $M$. 

Since we fix the operator $\Phi$ throughout the article, as long as the norm concerned is finite, we omit the dependence of the bounds on $\|\Phi\|$ for simplicity. We also omit the dependence on $\rho_0 \geq 5$ since it is also fixed. On the other hand, to construct a solution in $L_{\omega}^{\rho_0}$, we will need certain bounds in a higher integrability space $L_{\omega}^{\rho}$ for $\rho > \rho_0$. We will point out the dependence on this larger $\rho$ when it appears below. 

\subsection{Overview of the proof}

The proof consists of several ingredients, all of which use bootstrap argument over smaller subintervals of $[0,1]$. In order to show $\{u_{m,\eps}\}_{\eps}$ is Cauchy, we need to control the difference $u_{m,\eps_1} - u_{m,\eps_2}$ for small $\eps_1$ and $\eps_2$ over those subintervals and then iterate. Even though the two processes start with the same initial data, they start to differ instantly after the evolution begins. Hence, in order to be able to iterate over subintervals, we need the solution to be stable under perturbation of initial data. This is the following proposition.

\begin{prop}[Uniform stability in $\eps$] \label{pr:uniform_stable}
	Let $M>0$ be arbitrary. Let $\uem$ and $\vem$ denote the solutions to \eqref{eq:trun_sub} with initial datum $u_0$ and $v_0$ respectively. Suppose
	\begin{equation*}
	\|u_0\|_{L_{\omega}^{\infty}L_{x}^{2}} \leq M\;, \qquad \|v_0\|_{L_{\omega}^{\infty}L_{x}^{2}} \leq M. 
	\end{equation*}
	Then we have the bound
	\begin{equation*}
	\|u_{m,\eps}-v_{m,\eps}\|_{L_{\omega}^{\rho_0} \xX(0,1)} < C \|u_0 - v_0\|_{L_{\omega}^{\rho_0} L_{x}^{2}}
	\end{equation*}
	for some constant $C$ depending on $\rho_0$, $m$ and $M$ only. 
\end{prop}

The above proposition compares two solutions to the same equation with different initial data. On the other hand, in order to show $\{u_{m,\eps}\}$ is Cauchy in $\eps$, we need to compare $\nN^{\eps_1}(u_{m,\eps_1})$ and $\nN^{\eps_2}(u_{m,\eps_2})$ for $\eps_1 \neq \eps_2$. Since the two nonlinearities carry different powers, we need a priori bound on $L_{x}^{\infty}$ norm of $u_{m,\eps_j}(t)$ to get effective control on their difference. This requires the initial data to be in a more regular space than $L_{x}^{2}$. Hence, we make a small perturbation of initial data to $H_{x}^{1}$. The following proposition guarantees that the solution will still be in $H_{x}^{1}$ up to time $1$.

\begin{prop}[Persistence of regularity]
	\label{pr:persistence}
	Let $\rho \geq 5$ and $v_0 \in L_{\omega}^{\rho}H_{x}^{1}$. Then for every $m,\eps>0$, there exists $v_{m,\eps} \in L_{\omega}^{\rho}\xX^{1}(0,1)$ such that \eqref{eq:duhamel_trun_sub} holds with initial data $v_0$ and in the same space. Furthermore, we have the bound
	\begin{equation*}
	\|v_{m,\eps}\|_{L_{\omega}^{\rho}\xX^{1}(0,1)} \leq C \|v_0\|_{L_{\omega}^{\rho}H_{x}^{1}}, 
	\end{equation*}
	where $C$ depends on $\rho$ and $m$ only. 
\end{prop}

\begin{rmk}
	Note that the proposition is stated with arbitrary $\rho \geq 5$, not just our fixed $\rho_0$. This is because later when we show the convergence in $L_{\omega}^{\rho_0}\xX(0,1)$, we actually need a higher integrability than $\rho_0$ (see Proposition~\ref{pr:convergence_high} below). 
\end{rmk}

\begin{rmk}
Propositions~\ref{pr:uniform_stable} and~\ref{pr:persistence} can be viewed as natural generalisations of \cite[Lemma~3.10, 3.12]{colliander2008global}. 
\end{rmk}

Thanks to the persistence of regularity, we can show the convergence of the solutions when starting from $L_{\omega}^{\infty} H_{x}^{1}$ initial data. This is the following proposition. 

\begin{prop} [Convergence with regular initial data]
	\label{pr:convergence_high}
	Let $v_{m,\eps}$ denotes the solutions to \eqref{eq:trun_sub} with initial data $v_0 \in L_{\omega}^{\infty}H_{x}^{1}$. Then for every $m>0$,  $\{v_{m,\eps}\}_{\eps}$ is Cauchy in $L_{\omega}^{\rho_0} \xX(0,1)$. 
\end{prop}

The three propositions above are all the ingredients we need. We will prove them in the next three subsections, and combine them together to prove Proposition~\ref{pr:converge_e} in the last subsection.

\subsection{Uniform stability -- proof of Proposition~\ref{pr:uniform_stable}}

The key ingredient to prove Proposition~\ref{pr:uniform_stable} is the following lemma.

\begin{lem} \label{le:stability_short}
	There exist $h, C>0$ depending on $\rho_0$, $m$ and $M$ only such that
	\begin{equation*}
	\|\uem - \vem\|_{L_{\omega}^{\rho_0} \xX(0,b)} \leq C \|u_{m,\eps} - v_{m,\eps}\|_{L_{\omega}^{\rho_0} \xX(0,a)}
	\end{equation*}
	whenever $[a,b] \subset [0,1]$ satisfies $b-a < h$. The bound is uniform over all $\eps \in (0,1)$. 
\end{lem}
\begin{proof}
	Fix an arbitrary interval $[a,b] \subset [0,1]$ with $b-a<h$, where $h \leq 1$ will be specified later. For every $\omega \in \Omega$, we choose a dissection $\{\tau_k\}$ of the interval $[a,b]$ as follows. Let $\tau_0 = a$. Suppose $a = \tau_0 < \cdots < \tau_{k} < b$ is chosen, we choose $\tau_{k+1}$ by
	\begin{equation} \label{eq:stability_dissection}
	\begin{split}
	\tau_{k+1} := b \wedge \inf \Big\{ &r>\tau_k: \1_{\big\{\|u_{m,\eps}\|_{\xX_{2}(0,\tau_k)}^{5} < 2m\big\}} \|u_{m,\eps}\|_{\xX_{2}(\tau_k,r)}^{5}\\
	&+ \1_{\big\{\|v_{m,\eps}\|_{\xX_{2}(0,\tau_k)}^{5} < 2m\big\}} \|v_{m,\eps}\|_{\xX_{2}(\tau_k,r)}^{5} \geq \eta \Big\}, 
	\end{split}
	\end{equation}
	where $\eta>0$ is a small number to be specified later. In this way, we get a random dissection
	\begin{equation*}
	a = \tau_0 < \tau_1 < \cdots < \tau_K = b. 
	\end{equation*}
	Note that the total number $K$ of subintervals is always bounded by
	\begin{equation} \label{eq:stability_no_interval}
	K \leq 2 \times \Big(1 + \frac{4m}{\eta} \Big). 
	\end{equation}
	Let $\iI_{k+1} = [\tau_{k}, \tau_{k+1}]$. For every $k = 0, \dots, K-1$ and every $t \in \iI_{k+1}$, we have
	\begin{equation*}
	\begin{split}
	&u_{m,\eps}(t) - v_{m,\eps}(t) = e^{i (t-\tau_k) \Delta} \big( u_{m,\eps}(\tau_k) - v_{m,\eps}(\tau_k) \big) - \dD_{\tau_k}(t) \\
	&- i \int_{\tau_k}^{t} \sS(t-s) \big( u_{m,\eps}(s) - v_{m,\eps}(s) \big) {\rm d} W_s - \frac{1}{2} \int_{\tau_k}^{t} \sS(t-s) \Big( F_{\Phi} \big( u_{m,\eps}(s) - v_{m,\eps}(s) \big) \Big) {\rm d}s, 
	\end{split}
	\end{equation*}
	where
	\begin{equation*}
	\begin{split}
	\dD_{\tau}(t) = - i \int_{\tau}^{t} &\sS(t-s) \Big( \theta_{m}\big( \|u_{m,\eps}\|_{\xX_{2}(0,s)}^{5} \big) \nN^{\eps} \big(u_{m,\eps}(s)\big)\\
	&- \theta_{m}\big( \|v_{m,\eps}\|_{\xX_{2}(0,s)}^{5} \big) \nN^{\eps} \big(v_{m,\eps}(s)\big) \Big) {\rm d}s. 
	\end{split}
	\end{equation*}
	We now control the $\xX(\iI_{k+1})$ norm of the four terms on the right hand side separately. For the term with the initial data, it follows immediately from the Strichartz estimates \eqref{eq:Strichartz_1} that
	\begin{equation} \label{eq:stability_initial}
	\big\|e^{i (t-\tau_k) \Delta} \big( u_{m,\eps}(\tau_k) - v_{m,\eps}(\tau_k) \big) \big\|_{\xX(\iI_{k+1})} \leq C \|u_{m,\eps}-v_{m,\eps}\|_{\xX(0,\tau_k)}, 
	\end{equation}
	where we have enlarged $\|\cdot\|_{L_{x}^{2}}$ to $\|\cdot\|_{\xX}$. For the It\^o-Stratonovich correction term, by Proposition~\ref{pr:pre_correction}, we have
	\begin{equation} \label{eq:stability_correction}
	\Big\| \int_{\tau_k}^{t} \sS(t-s) \Big( F_{\Phi} \big( u_{m,\eps}(s) - v_{m,\eps}(s) \big) \Big) {\rm d}s \Big\|_{\xX(\iI_{k+1})} \leq C h^{\frac{4}{5}} \|u_{m,\eps} - v_{m,\eps}\|_{\xX(\iI_{k+1})}. 
	\end{equation}
	We only have $h^{\frac{4}{5}}$ on the right hand side but not $h$ since $h \leq 1$. As for the stochastic term, we let $\mM_{1,a}^{*}$ and $\mM_{2,a}^{*}$ be the same as \eqref{eq:martingales} with $\sigma = u_{m,\eps} - v_{m,\eps}$. We then have
	\begin{equation} \label{eq:stability_stochastic}
	\begin{split}
	&\phantom{11}\sup_{k}\Big\|\int_{\tau_k}^{t} \sS(t-s) \big(u_{m,\eps}(s)-v_{m,\eps}(s)\big) {\rm d} W_s\Big\|_{\xX(\iI_{k+1})}\\
	&\leq \|\mM_{1,a}^{*}\|_{L_{t}^{\infty}(a,b)} + \|\mM_{2,a}^{*}\|_{L_{t}^{5}(a,b)} =: \mM_{a,b}^{*}. 
	\end{split}
	\end{equation}
	By Proposition~\ref{pr:pre_stochastic} and that $b-a \leq h \leq 1$, we have the bound
	\begin{equation} \label{eq:stability_stochastic_bound}
	\|\mM_{a,b}^{*}\|_{L_{\omega}^{\rho_0}} \leq C h^{\frac{3}{10}} \|u_{m,\eps} - v_{m,\eps}\|_{L_{\omega}^{\rho_0} \xX(a,b)}. 
	\end{equation}
	We finally turn to the nonlinearity $\dD_{\tau_k}$. For this, we consider situations depending on whether the quantities $\|u_{m,\eps}\|_{\xX_{2}(0,\tau_k)}^{5}$ and $\|v_{m,\eps}\|_{\xX_{2}(0,\tau_k)}^{5}$ have reached $2m$ or not. 
	
	\begin{flushleft}
		\textit{Situation 1.}
	\end{flushleft}
	We first consider the situation when both $\|\uem\|_{\xX_2(0,\tau_k)}^{5}$ and $\|\vem\|_{\xX_2(0,\tau_k)}^{5}$ are smaller than $2m$. In this case, we bound the integrand in $\dD_{\tau_k}(t)$ pointwise by
	\begin{equation} \label{eq:stability_nonlinear_pt}
	\begin{split}
	&\phantom{111}\Big| \theta_{m}\big(\|\uem\|_{\xX_{2}(0,s)}^{5}\big) \nN^{\eps}\big(u_{m,\eps}(s)\big) - \theta_{m}\big(\|\vem\|_{\xX_{2}(0,s)}^{5}\big) \nN^{\eps}\big(v_{m,\eps}(s)\big) \Big|\\
	&\leq \Big| \theta_{m}\big(\|\uem\|_{\xX_2(0,s)}^{5}\big) \nN^{\eps}\big( \uem(s) \big) - \theta_{m}\big(\|\vem\|_{\xX_2(0,s)}^{5}\big) \nN^{\eps}\big(\uem(s)\big) \Big|\\
	&\phantom{11}+ \Big| \theta_{m}\big(\|\vem\|_{\xX_2(0,s)}^{5}\big) \nN^{\eps}\big(u_{m,\eps}(s)\big) - \theta_{m}\big(\|\vem\|_{\xX_2(0,s)}^{5}\big) \nN^{\eps}\big(v_{m,\eps}(s)\big) \Big|\\
	&\leq C \Big( \|u_{m,\eps} - v_{m,\eps}\|_{\xX_{2}(0,\tau_{k+1})} \big( \|\uem\|_{\xX_2(0,\tau_{k+1})}^{4} + \|\vem\|_{\xX_2(0,\tau_{k+1})}^{4} \big) |u_{m,\eps}(s)|^{5-\eps}\\
	&\phantom{11}+ |\uem(s)-\vem(s)| \big( |\uem(s)|^{4-\eps} + |\vem(s)|^{4-\eps} \big)\Big), 
	\end{split}
	\end{equation}
	where we have made the relaxation
	\begin{equation*}
	\begin{split}
	&\phantom{1111}\big| \|u_{m,\eps}\|_{\xX_{2}(0,s)}^{5} - \|v_{m,\eps}\|_{\xX_{2}(0,s)}^{5} \big|\\
	&\leq C \big| \|\uem\|_{\xX_2(0,s)} - \|\vem\|_{\xX_2(0,s)} \big| \big( \|\uem\|_{\xX_2(0,s)}^{4} + \|\vem\|_{\xX_2(0,s)}^{4} \big) \\
	&\leq C\|u_{m,\eps} - v_{m,\eps}\|_{\xX_{2}(0,\tau_{k+1})} \big( \|\uem\|_{\xX_2(0,\tau_{k+1})}^{4} + \|\vem\|_{\xX_2(0,\tau_{k+1})}^{4} \big)
	\end{split}
	\end{equation*}
	in the last inequality. By the assumption of this situation as well as the choice of dissection in \eqref{eq:stability_dissection}, we have
	\begin{equation} \label{eq:stability_as_1}
	\|\uem\|_{\xX_2(\iI_{k+1})} \leq \eta^{\frac{1}{5}}, \quad \|\uem\|_{\xX_2(0,\tau_{k+1})} \leq (2m)^{\frac{1}{5}} + \eta^{\frac{1}{5}}, \quad \|\uem\|_{\xX_1(\iI_{k+1})} \leq M, 
	\end{equation}
	and the same is true for $\vem$. Hence, assuming $\eta<1$, applying the bound \eqref{eq:pre_nonlinear_eps} in Proposition~\ref{pr:pre_nonlinear} and the pointwise bound \eqref{eq:stability_nonlinear_pt}, and using \eqref{eq:stability_as_1}, we get
	\begin{equation*}
	\|\dD_{\tau_k}\|_{\xX(\iI_{k+1})} \leq C_{m,M} \cdot \eta^{\frac{4-\eps}{5}} \Big( \|\uem-\vem\|_{\xX(0,\tau_k)} + \|\uem-\vem\|_{\xX(\iI_{k+1})} \Big)
	\end{equation*}
	where we have split $\|u_{m,\eps}-v_{m,\eps}\|_{\xX(0,\tau_{k+1})}$ into two disjoint intervals $[0,\tau_k]$ and $\iI_{k+1}$. 
	
	\begin{flushleft}
		\textit{Situation 2.}
	\end{flushleft}
    We turn to the situation when $\|u_{m,\eps}\|_{\xX_{2}(0,\tau_k)}^{5} < 2m$ but $\|v_{m,\eps}\|_{\xX_{2}(0,\tau_k)}^{5} \geq 2m$. In this case, $\theta_{m}(\|v_{m,\eps}\|_{\xX_{2}(0,s)}^{5})$ vanishes for every $s>\tau_k$. Hence, we have the pointwise bound
    \begin{equation*}
    \begin{split}
    &\phantom{111}\Big| \theta_{m} \big(\|u_{m,\eps}\|_{\xX_2(0,s)}^{5} \big) \nN^{\eps}\big(u_{m,\eps}(s)\big) - \theta_{m} \big(\|v_{m,\eps}\|_{\xX_2(0,s)}^{5} \big) \nN^{\eps}\big(v_{m,\eps}(s)\big) \Big|\\
    &= \Big| \theta_{m} \big(\|u_{m,\eps}\|_{\xX_2(0,s)}^{5} \big) \nN^{\eps}\big(u_{m,\eps}(s)\big) - \theta_{m} \big(\|v_{m,\eps}\|_{\xX_2(0,s)}^{5} \big) \nN^{\eps}\big(u_{m,\eps}(s)\big) \Big|\\
    &\leq C \|u_{m,\eps} - v_{m,\eps}\|_{\xX_{2}(0,\tau_{k+1})} \big( \|\uem\|_{\xX_2(0,\tau_{k+1})}^{4} + \|\vem\|_{\xX_2(0,\tau_{k+1})}^{4} \big) |u_{m,\eps}(s)|^{5-\eps}
    \end{split}
    \end{equation*}
    for $s \in \iI_{k+1}$. Similar as before, applying the first bound in Proposition~\ref{pr:pre_nonlinear} to $\dD_{\tau_k}$ and using the above pointwise bound as well as the dissection \eqref{eq:stability_dissection}, we again get
    \begin{equation*}
    \|\dD_{\tau_k}\|_{\xX(\iI_{k+1})} \leq C_{m,M} \cdot \eta^{\frac{4-\eps}{5}} \Big( \|\uem-\vem\|_{\xX(0,\tau_k)} + \|\uem-\vem\|_{\xX(\iI_{k+1})} \Big). 
    \end{equation*}
    The right hand side above is symmetric in $\uem$ and $\vem$, so we have exactly the same bound in the case $\|\uem\|_{\xX_{2}(0,\tau_k)} \geq 2m$ but $\|\vem\|_{\xX_{2}(0,\tau_k)} < 2m$. 
    
    \begin{flushleft}
    	\textit{Situation 3.}
    \end{flushleft}
    If both $\|u_{m,\eps}\|_{\xX_{2}(0,\tau_k)}^{5}$ and $\|v_{m,\eps}\|_{\xX_{2}(0,\tau_k)}^{5}$ reaches $2m$, then the nonlinearity vanishes, and we have $\|\dD_{\tau_k}\|_{\xX(\iI_{k+1})} = 0$. 
    
    \bigskip
    
    \bigskip
    
    Since the above three situations include all possibilities, we always have the bound
    \begin{equation} \label{eq:stability_nonlinear}
    \|\dD_{\tau_k}\|_{\xX(\iI_{k+1})} \leq C_{m,M} \cdot \eta^{\frac{4-\eps}{5}} \Big( \|\uem-\vem\|_{\xX(0,\tau_k)} + \|\uem-\vem\|_{\xX(\iI_{k+1})} \Big). 
    \end{equation}
    Combining \eqref{eq:stability_initial}, \eqref{eq:stability_correction}, \eqref{eq:stability_stochastic} and \eqref{eq:stability_nonlinear}, we get
    \begin{equation*}
    \begin{split}
    \|u_{m,\eps} - v_{m,\eps}\|_{\xX(\iI_{k+1})} \leq C_{m,M} &\Big( \big( h^{\frac{4}{5}} + \eta^{\frac{4-\eps}{5}} \big) \|u_{m,\eps} - v_{m,\eps}\|_{\xX(\iI_{k+1})}\\
    &+ \|u_{m,\eps}-v_{m,\eps}\|_{\xX(0,\tau_k)} \Big) + \mM_{a,b}^{*}. 
    \end{split}
    \end{equation*}
    If both $\eta$ and $h$ are small enough (depending on $m$ and $M$ only, but independent of $\eps$), we can absorb the term $\|u_{m,\eps}-v_{m,\eps}\|_{\xX(\iI_{k+1})}$ into the left hand side. Then, adding $\|u_{m,\eps}-v_{m,\eps}\|_{\xX(0,\tau_k)}$ to both sides above, we get
    \begin{equation*}
    \|u_{m,\eps} - v_{m,\eps}\|_{\xX(0,\tau_{k+1})} \leq C_{m,M} \Big( \|u_{m,\eps} - v_{m,\eps}\|_{\xX(0,\tau_k)} + \mM_{a,b}^{*} \Big). 
    \end{equation*}
    Since $\tau_0 = a$ and $\tau_K = b$, iterating the above bound $K$ times, we obtain
    \begin{equation} \label{eq:stability_deterministic}
    \|u_{m,\eps}-v_{m,\eps}\|_{\xX(0,b)} \leq C_{m,M}^{K} \Big( \|u_{m,\eps}-v_{m,\eps}\|_{\xX(0,a)} + \mM_{a,b}^{*} \Big), 
    \end{equation}
    where $K$ is at most $2 \times \big(1 + 4m/\eta \big)$. So far, all the arguments are deterministic, and the bound \eqref{eq:stability_deterministic} holds almost surely. Moreover, the proportionality constant $C_{m,M}^{K}$ above is deterministic, and depends on $m$ and $M$ only since $\eta$ does. 
    
    We now take $L_{\omega}^{\rho_0}$ norm on both sides of \eqref{eq:stability_deterministic}. By \eqref{eq:stability_stochastic_bound}, if $h$ is small enough depending on $m$ and $M$, we can again absorb the term $\|u_{m,\eps}-v_{m,\eps}\|_{L_{\omega}^{\rho_0}\xX(a,b)}$ arising from $\|\mM_{a,b}^{*}\|_{L_{\omega}^{\rho_0}}$ into the left hand side. Hence, we finally obtain the bound
    \begin{equation*}\|u_{m,\eps} - v_{m,\eps}\|_{L_{\omega}^{\rho_0}\xX(0,b)} \leq C_{m,M} \|u_{m,\eps} - v_{m,\eps}\|_{L_{\omega}^{\rho_0}\xX(0,a)}, 
    \end{equation*}
    thus completing the proof. 
\end{proof}

We are now ready to prove Proposition~\ref{pr:uniform_stable}. 

\begin{proof} [Proof of Proposition~\ref{pr:uniform_stable}]
	The key is to note that the smallness of $h$ in order for Lemma~\ref{le:stability_short} to be true depends only on $\rho_0$, $m$, and $M$ only. We can then iterate Lemma~\ref{le:stability_short} with small intervals of length $h$ up to time $1$. This completes the proof. 
\end{proof}

\subsection{Persistence of  regularity -- proof of Proposition~\ref{pr:persistence}}

In the statement of Proposition~\ref{pr:persistence}, we need a bound for higher integrability than $\rho_0$, so we deal with arbitrary $\rho \geq 5$. All the bounds below depend on $\rho$, and we omit this dependence in notation for simplicity. Recall that the $\|\cdot\|_{\xX^1(\iI)}$ norm is given by
\begin{equation*}
\|v\|_{\xX^1(\iI)} := \|v\|_{\xX(\iI)} + \|\d_x v\|_{\xX(\iI)}. 
\end{equation*}
The following local existence of the $H_x^1$ solution is standard, and can be obtained directly via a contraction argument. 

\begin{lem} \label{le:H1_local}
	Let $v_0 \in L_{\omega}^{\rho}H_{x}^{1}$. Then there exists $R_0 > 0$ depending on $\rho$ and $\|v_0\|_{L_{\omega}^{\rho}H_{x}^{1}}$ only such that for every $m,\eps>0$, there is a unique $v_{m,\eps} \in L_{\omega}^{\rho}\xX^1(0,R_0)$ that solves \eqref{eq:duhamel_trun_sub} in the same space with initial data $v_0$. 
\end{lem}

Our aim is to show that with initial data in $L_{\omega}^{\rho}H_{x}^{1}$, the solution actually exists in $[0,1]$ (if $R_0<1$) and satisfies persistence of regularity. The key ingredient is the following lemma. 

\begin{lem} \label{le:persistence_short}
	There exist $h, C>0$ depending on $\rho$ and $m$ such that if $v_{m,\eps}$ solves \eqref{eq:trun_sub} in $L_{\omega}^{\rho} \xX^{1}(0,R)$ with initial data $v_0 \in L_{\omega}^{\rho}H_{x}^{1}$, then we have the bound
	\begin{equation} \label{eq:persistence_short}
	\|v_{m,\eps}\|_{L_{\omega}^{\rho}\xX^{1}(0,b)} \leq C \|v_{m,\eps}\|_{L_{\omega}^{\rho} \xX^{1}(0,a)}
	\end{equation}
	whenever $[a,b] \subset [0,R]$ satisfies $b-a<h$. As a consequence, we have
	\begin{equation} \label{eq:persistence_medium}
	\|v_{m,\eps}\|_{L_{\omega}^{\rho} \xX^{1}(0,R)} \leq C^{1+R} \|v_0\|_{L_{\omega}^{\rho}H_{x}^{1}}. 
	\end{equation}
	Here, the constant $C$ depends on $m$ and $\rho$ but is uniform over all $\eps \in (0,1)$ and all $R>0$. 
\end{lem}
\begin{proof}
	The proof is essentially the same as that for Lemma~\ref{le:stability_short}. Let $[a,b] \subset [0,R]$ with $b-a<h$, where $h$ small is to be specified later. We choose a random dissection $\{\tau_k\}_{k=0}^{K}$ of the interval $[a,b]$ as follows. Let $\tau_0=a$, and define
	\begin{equation*}
	\tau_{k+1} = b \wedge \inf \big\{r > \tau_k: \1_{\big\{\|v_{m,\eps}\|_{\xX_2(0,\tau_k)}^{5} \leq 2m\big\}} \|v_{m,\eps}\|_{\xX_2(\tau_k,r)}^{5} \geq \eta \big\}
	\end{equation*}
	for some $\eta$ to be specified later. The total number of intervals $K$ is at most $1 + \frac{2m}{\eta}$. Let $\iI_{k+1} = [\tau_{k}, \tau_{k+1}]$. On $\iI_{k+1}$, $v_{m,\eps}$ satisfies the Duhamel formula
	\begin{equation*}
	\begin{split}
	v_{m,\eps}(t) = &e^{i(t-\tau_k)\Delta} v_{m,\eps}(\tau_k) - i \int_{\tau_k}^{t} \sS(t-s) \Big( \theta_{m}\big( \|v_{m,\eps}\|_{\xX_2(0,s)}^{5} \big) \nN^{\eps}\big(v_{m,\eps}(s)\big) \Big) {\rm d}s\\
	&-i \int_{\tau_k}^{t} \sS(t-s) v_{m,\eps}(s) {\rm d} W_s - \frac{1}{2} \int_{\tau_k}^{t} \sS(t-s) \big( F_{\Phi} v_{m,\eps}(s) \big) {\rm d}s. 
	\end{split}
	\end{equation*}
	We need to control the $\xX^{1}$-norm of $v_{m,\eps}$ on $\iI_{k+1}$, which involves $v_{m,\eps}$ and its spatial derivative. The key is that the differentiation in $x$ variable commutes with the operator $\sS(t-s)$. Hence, for the initial data and the correction term, by the Strichartz estimates \eqref{eq:Strichartz_1} and Proposition~\ref{pr:pre_correction}, we have the pathwise bounds
	\begin{equation} \label{eq:persistence_initial_correction}
	\begin{split}
	\|e^{i(t-\tau_k)\Delta} v_{m,\eps}(\tau_k)\|_{\xX^{1}(\iI_{k+1})} \leq C \|v_{m,\eps}(\tau_k)\|_{H_{x}^{1}} \leq C \|v_{m,\eps}\|_{\xX^{1}(0,\tau_k)}, \\
	\Big\|\int_{\tau_k}^{t} \sS(t-s) \big( F_{\Phi} v_{m,\eps}(s) \big) {\rm d}s \Big\|_{\xX^{1}(\iI_{k+1})} \leq C h^{\frac{4}{5}} \|v_{m,\eps}\|_{\xX^{1}(\iI_{k+1})}. 
	\end{split}
	\end{equation}
	Here, both constants $C$ are deterministic and universal. As for the stochastic term, we let
	\begin{equation*}
	\begin{split}
	\mM_{1}^{**}(t) &= \sup_{a \leq r_1 \leq r_2 \leq t} \Big\| \int_{r_1}^{r_2} \sS(t-s) v_{m,\eps}(s) {\rm d} W_s \Big\|_{H_{x}^{1}}, \\
	\mM_{2}^{**}(t) &= \sup_{a \leq r_1 \leq r_2 \leq t} \Big\| \int_{r_1}^{r_2} \sS(t-s) v_{m,\eps}(s) {\rm d} W_s \Big\|_{W_{x}^{1,10}}, 
	\end{split}
	\end{equation*}
	where $\|f\|_{W_{x}^{1,10}} = \|f\|_{L_{x}^{10}} + \|\d_x f\|_{L_{x}^{10}}$. We have
	\begin{equation} \label{eq:persistence_stochastic}
	\Big\|\int_{\tau_k}^{t} \sS(t-s) v_{m,\eps}(s) {\rm d}W_s \Big\|_{\xX^{1}(\iI_{k+1})} \leq \|\mM_{1}^{**}\|_{L_{t}^{\infty}(a,b)} + \|\mM_{2}^{**}\|_{L_{t}^{5}(a,b)} =: \mM_{a,b}^{**}. 
	\end{equation}
	To get an upper bound for $\mM_{a,b}^{**}$, we need to control the $\xX_1$ and $\xX_2$ norms of the following three quantities: 
	\begin{equation*}
	\int_{r_1}^{r_2} \sS(t-s) v_{m,\eps} {\rm d} W_s, \phantom{1} \int_{r_1}^{r_2} \sS(t-s) \d_x v_{m,\eps}(s) {\rm d} W_s, \phantom{1} \int_{r_1}^{r_2} \sS(t-s) v_{m,\eps}(s) {\rm d} \d_x W_s.  
	\end{equation*}
	The first two terms can be treated directly with Proposition~\ref{pr:pre_stochastic}. As for the third term, the only difference is that the noise $W$ is replaced by $\d_x W$. This amounts to replace $\|\Phi\|$ in that proposition by $\|\d_x \circ \Phi\|$ with the same norm, which is also finite by Assumption~\ref{as:Phi} and Lemma~\ref{le:factorisation}. Hence, it follows from Proposition~\ref{pr:pre_stochastic} that
	\begin{equation} \label{eq:persistence_stochastic_bound}
	\|\mM_{a,b}^{**}\|_{L_{\omega}^{\rho}} \leq C_{\rho} h^{\frac{3}{10}} \|v_{m,\eps}\|_{L_{\omega}^{\rho} \xX^{1}(a,b)}. 
	\end{equation}
	Finally, for the nonlinear term, we distinguish two cases depending on whether $\|v_{m,\eps}\|_{\xX_2(0,\tau_k)}^{5}$ reaches $2m$ or not. If $\|v_{m,\eps}\|_{\xX_2(\tau_k)}^{5} < 2m$, then the choice of $\{\tau_k\}$ guarentees that $\|v_{m,\eps}\|_{\xX_2(\iI_{k+1})} \leq \eta^{\frac{1}{5}}$. Since
	\begin{equation*}
	|\d_x \nN^{\eps}(v_{m,\eps})| \leq C |\d_x v_{m,\eps}| \cdot |v_{m,\eps}|^{4-\eps}, 
	\end{equation*}
	exchanging $\d_x$ and $\sS(t-s)$ and applying Proposition~\ref{pr:pre_nonlinear}, we get the bound
	\begin{equation} \label{eq:persistence_nonlinear}
	\Big\| \int_{\tau_k}^{t} \sS(t-s)\Big( \theta_{m}\big(\|v_{m,\eps}\|_{\xX_2(0,s)}^{5}\big) \nN^{\eps}\big( v_{m,\eps}(s) \big) \Big) {\rm d}s \Big\|_{\xX^1(\iI_{k+1})} \leq C \eta^{\frac{4-\eps}{5}} \|v_{m,\eps}\|_{\xX^{1}(\iI_{k+1})}
	\end{equation}
	if $\|v_{m,\eps}\|_{\xX_2(0,\tau_{k})} < 2m$. If $\|v_{m,\eps}\|_{\xX_2(0,\tau_{k+1})} \geq 2m$, then the nonlinearity vanishes on $\iI_{k+1}$, and hence the above bound is also true. Combining \eqref{eq:persistence_initial_correction}, \eqref{eq:persistence_stochastic} and \eqref{eq:persistence_nonlinear}, we have 
	\begin{equation*}
	\|v_{m,\eps}\|_{\xX^{1}(\iI_{k+1})} \leq C \big( \|v_{m,\eps}\|_{\xX^{1}(0,\tau_k)} + (\eta^{\frac{4-\eps}{5}}+h^{\frac{4}{5}}) \|v_{m,\eps}\|_{\xX^{1}(\iI_{k+1})} \big) + \mM_{a,b}^{**}, 
	\end{equation*}
	where the constant $C$ is universal.  If both $h$ and $\eta$ are sufficiently small (and independent of any parameter), we can absorb the term $\|v_{m,\eps}\|_{\xX^{1}(\iI_{k+1})}$ into the left hand side. Then adding $\|v_{m,\eps}\|_{\xX^{1}(0,\tau_k)}$ to both sides, we get
	\begin{equation*}
	\|v_{m,\eps}\|_{\xX^{1}(0,\tau_{k+1})} \leq C \big( \|v_{m,\eps}\|_{\xX^{1}(0,\tau_k)} + \mM_{a,b}^{**} \big). 
	\end{equation*}
	Since the number of intervals $K$ is at most $1+\frac{2m}{\eta}$, iterating the above bound over $k=0, \dots, K-1$ gives
	\begin{equation*}
	\|v_{m,\eps}\|_{\xX^{1}(0,b)} \leq C_{m} \big( \|v_{m,\eps}\|_{\xX^{1}(0,a)} + \mM_{a,b}^{**} \big). 
	\end{equation*}
	Now, we take $L_{\omega}^{\rho}$-norm on both sides. By \eqref{eq:persistence_stochastic_bound}, if $h$ is sufficiently small (but depending on $\rho$ and $m$ this time), we can absorb the term $\|\mM_{a,b}^{**}\|_{L_{\omega}^{\rho}}$ to the left hand side to obtain \eqref{eq:persistence_short}. 
	
	As for the second part of the claim, we iterate the bound \eqref{eq:persistence_short} in the interval $[0,R]$ for at most $1 + \frac{R}{h}$ steps. Since the smallness of $h$ for \eqref{eq:persistence_short} to be true depends on $\rho$ and $m$ only, the bound \eqref{eq:persistence_medium} then follows. 
\end{proof}

\begin{proof} [Proof of Proposition~\ref{pr:persistence}]
	By Lemma~\ref{le:H1_local}, we know there exists a unique $v_{m,\eps}$ that satisfies the integral equation \eqref{eq:duhamel_trun_sub} in $L_{\omega}^{\rho_0} \xX^{1}(0,R_0)$ with initial data $v_0$, where $R_0>0$ depends on $\|v_0\|_{L_{\omega}^{\rho_0}H_{x}^{1}}$ only. If $R_0 \geq 1$, then Lemma~\ref{le:persistence_short} immediately implies the desired bound. So we only need to consider the case $R_0 < 1$. 
	
	Let $C_{\rho,m}^{*} = (1+C)^{2}$ for the same $C$ as in \eqref{eq:persistence_medium}. We have replaced $R$ by $1$ since we only need to control the solution up to time $1$. As the notation suggests, $C_{\rho,m}^{*}$ depends on $\rho$ and $m$ only, but not on $R_0$. If $R_0 < 1$, then by Lemma~\ref{le:persistence_short}, we have
	\begin{equation} \label{eq:persistence_1}
	\|v_{m,\eps}(R_0/2)\|_{L_{\omega}^{\rho}H_{x}^{1}} \leq C_{\rho,m}^{*} \|v_0\|_{L_{\omega}^{\rho} H_{x}^{1}}, 
	\end{equation}
	Again by Lemma~\ref{le:H1_local}, starting from $\frac{R_0}{2}$, we can extend the solution up to time $\frac{R_0}{2} + T_0$, where $T_0$ depends on $C_{\rho,m}^{*} \|v_0\|_{L_{\omega}^{\rho_0}H_{x}^{1}}$ only. Lemma~\ref{le:persistence_short} implies the bound
	\begin{equation} \label{eq:persistence_2}
	\|v_{m,\eps}\|_{L_{\omega}^{\rho} \xX^{1}(0,1 \wedge (\frac{R_0}{2}+T_0))} \leq C_{\rho,m}^{*} \|v_0\|_{L_{\omega}^{\rho}H_{x}^{1}}, 
	\end{equation}
	where $C_{\rho,m}^{*}$ is the same as in \eqref{eq:persistence_1}. If $\frac{R_0}{2}+T_{0} > 1$, then the proof is finished. If not, we can again repeat the extension by $T_0$ while keeping the bound of the form \eqref{eq:persistence_2} with the same $C_{\rho,m}^{*}$. Hence, the process necessarily ends in finitely steps, and the same bound as \eqref{eq:persistence_1} always holds as long as the termininal time does not exceed $1$. this completes the proof of the proposition. 
\end{proof}

\subsection{Convergence with regular initial data -- proof of Proposition~\ref{pr:convergence_high}}

For $j=1,2$, let $v_{m,\eps_j}$ denotes the solution to \eqref{eq:trun_sub} with common initial data $v_0 \in L_{\omega}^{\infty} H_{x}^{1}$ and nonlinearity $\nN^{\eps_j}$. We suppose $\|v_0\|_{L_{\omega}^{\infty}H_{x}^{1}} \leq M$ (not just $L_{x}^2$). 

\begin{lem} \label{le:Cauchy_high_local}
	There exists $h=h(\rho_0,m,M)$ such that for every $\delta>0$ and every $v_0 \in L_{\omega}^{\infty}H_{x}^{1}$ with $\|v_0\|_{L_{\omega}^{\infty}H_{x}^{1}} \leq M$, there exist $\eps^{*}>0$ and $\kappa>0$ depending on $m$ and $M$ such that for every $\eps_1, \eps_2 < \eps^{*}$ and every $[a,b] \subset [0,1]$ with $b-a < h$, if
	\begin{equation*}
	\|v_{m,\eps_1} - v_{m,\eps_2}\|_{L_{\omega}^{\rho_0}\xX(0,a)} < \kappa, 
	\end{equation*}
	then we have the bound
	\begin{equation*}
	\|v_{m,\eps_1} - v_{m,\eps_2}\|_{L_{\omega}^{\rho_0}\xX(0,b)} < \delta. 
	\end{equation*}
\end{lem}
\begin{proof}
	The lemma involves comparison between two nonlinearities with different powres, which needs the information of $L_{x}^{\infty}$-norm of $v_{m,\eps_j}$. For every $\Lambda > 1$, let
	\begin{equation*}
	\Omega_{\Lambda} = \big\{ \omega \in \Omega: \|v_{m,\eps_1}\|_{L_{t}^{\infty}H_{x}^{1}(0,1)} \leq \Lambda,\; \|v_{m,\eps_2}\|_{L_{t}^{\infty}H_{x}^{1}(0,1)} \leq \Lambda \big\}. 
	\end{equation*}
	We split $\|v_{m,\eps_1} - v_{m,\eps_2}\|_{\xX(0,b)}$ into two parts, depending on whether $\omega \in \Omega_{\Lambda}$ or not. We start with $\Omega_{\Lambda}^{c}$. 
	
	By Proposition~\ref{pr:persistence}, there exists $C > 0$ depending on $m$, $M$ and $\rho_0$ such that
	\begin{equation*}
	\Pr (\Omega_{\Lambda}^{c}) \leq \frac{C}{\Lambda^{\rho_0}}
	\end{equation*}
	for every $\Lambda > 1$. Hence, using H\"older's inequality and Proposition~\ref{pr:persistence} with $\rho=2\rho_0$, we get
	\begin{equation} \label{eq:cauchy_average_1}
	\|(v_{m,\eps_1}-v_{m,\eps_2}) \1_{\Omega_{\Lambda}^{c}}\|_{L_{\omega}^{\rho_0} \xX(0,b)} \leq \|v_{m,\eps_1}-v_{m,\eps_2}\|_{L_{\omega}^{2\rho_0}\xX(0,b)} \big( \Pr(\Omega_{\Lambda}^{c}) \big)^{\frac{1}{2\rho_0}} \leq \frac{C}{\sqrt{\Lambda}}. 
	\end{equation}
	We now turn to $\|(v_{m,\eps_1}-v_{m,\eps_2}) \1_{\Omega_\Lambda}\|_{L_{\omega}^{\rho_0}\xX(0,b)}$. This part follows the same line of argument as that in Lemma~\ref{le:stability_short}, except the control of the nonlinearity is different since it involves two different powers. We first quickly go through the part that resembles Lemma~\ref{le:stability_short}. 
	
	For every $\omega \in \Omega$ (hence including $\Omega_\Lambda$ in particular), let
	\begin{equation*}
	a = \tau_0 < \tau_1 < \cdots < \tau_K = b
	\end{equation*}
	be the dissection of the interval $[a,b]$ chosen according to \eqref{eq:stability_dissection}, except that one replaces $u_{m,\eps}$ and $v_{m,\eps}$ by $v_{m,\eps_1}$ and $v_{m,\eps_2}$. With the same replacement, the bounds \eqref{eq:stability_initial}, \eqref{eq:stability_correction} and \eqref{eq:stability_stochastic} also carry straightforwardly, so for $\iI_{k+1} = [\tau_k, \tau_{k+1}]$, we have
	\begin{equation*}
	\begin{split}
	\|v_{m,\eps_1} - v_{m,\eps_2}\|_{\xX(\iI_{k+1})} &\leq C \Big( \|v_{m,\eps_1} - v_{m,\eps_2}\|_{\xX(0,\tau_k)} + h^{\frac{4}{5}} \|v_{m,\eps_1} - v_{m,\eps_2}\|_{\xX(\iI_{k+1})} \Big)\\
	&\phantom{11}+ \mM_{a,b}^{*} + \|\dD_{\tau_k}\|_{\xX(\iI_{k+1})}, 
	\end{split}
	\end{equation*}
	where
	\begin{equation*}
	\begin{split}
	\dD_{\tau_k}(t) &= - i \int_{\tau_k}^{t} \sS(t-s) \Big( \theta_{m}\big(\|v_{m,\eps_1}\|_{\xX_{2}(0,s)}\big) \nN^{\eps_1} \big(v_{m,\eps_1}(s)\big)\\
	&\phantom{11}- \theta_{m}\big(\|v_{m,\eps_2}\|_{\xX_{2}(0,s)}\big) \nN^{\eps_2} \big(v_{m,\eps_2}(s)\big) \Big) {\rm d}s, 
	\end{split}
	\end{equation*}
	and $\mM_{a,b}^{*} \geq 0$ is defined in the same way as \eqref{eq:stability_stochastic} and satisfies the moment bound
	\begin{equation} \label{eq:cauchy_stochastic}
	\|\mM_{a,b}^{*}\|_{L_{\omega}^{\rho_0}} \leq C h^{\frac{3}{10}} \|v_{m,\eps_1} - v_{m,\eps_2}\|_{L_{\omega}^{\rho_0} \xX(a,b)}. 
	\end{equation}
	Again, if $h$ is sufficiently small, we can absorb the term $\|v_{m,\eps_1} - v_{m,\eps_2}\|_{\xX(\iI_{k+1})}$ into the left hand side so that
	\begin{equation} \label{eq:cauchy_deter_1}
	\|v_{m,\eps_1} - v_{m,\eps_2}\|_{\xX(\iI_{k+1})} \leq C \Big( \|v_{m,\eps_1} - v_{m,\eps_2}\|_{\xX(0,\tau_k)} + \mM_{a,b}^{*} + \|\dD_{\tau_k}\|_{\xX(\iI_{k+1})} \Big). 
	\end{equation}
	So far it has been exactly the same as in Lemma~\ref{le:stability_short}, and the bound \eqref{eq:cauchy_deter_1} holds pathwise with a deterministic constant $C$. Now the difference comes in the bound for $\dD_{\tau_k}$, as the two nonlinearities have different powers, so we need information of the $L_{x}^{\infty}$ norms of $v_{m,\eps_j}$ to control their difference. This is why we need a solution in $H_x^1$. 
	
	We now start to control $\dD_{\tau_k}$. We split the integrand in $\dD_{\tau_k}$ by
	\begin{equation*}
	\begin{split}
	&\phantom{111}\theta_{m}\big(\|v_{m,\eps_1}\|_{\xX_2(0,s)}^{5}\big) \nN^{\eps_1}\big(v_{m,\eps_1}(s)\big) - \theta_{m}\big(\|v_{m,\eps_2}\|_{\xX_2(0,s)}^{5}\big) \nN^{\eps_2}\big(v_{m,\eps_2}(s)\big)\\
	&= \Big( \theta_{m}\big(\|v_{m,\eps_1}\|_{\xX_2(0,s)}^{5}\big) \nN^{\eps_1}\big(v_{m,\eps_1}(s)\big) - \theta_{m}\big(\|v_{m,\eps_2}\|_{\xX_2(0,s)}^{5}\big) \nN^{\eps_1}\big(v_{m,\eps_2}(s)\big) \Big)\\
	&+ \Big( \theta_{m}\big(\|v_{m,\eps_2}\|_{\xX_2(0,s)}^{5}\big) \nN^{\eps_1}\big(v_{m,\eps_2}(s)\big) - \theta_{m}\big(\|v_{m,\eps_2}\|_{\xX_2(0,s)}^{5}\big) \nN^{\eps_2}\big(v_{m,\eps_2}(s)\big) \Big),
	\end{split} 
	\end{equation*}
	and let $\dD_{\tau_k} = \dD_{\tau_k}^{(1)} + \dD_{\tau_k}^{(2)}$, corresponding to the two terms in the above decomposition of the nonlinearity respectively. For $\dD_{\tau_k}^{(1)}$, since the powers of the two nonlinearities are the same (both are $\eps_1$), we have exactly the same bound as in \eqref{eq:stability_nonlinear}, so that
	\begin{equation} \label{eq:cauchy_N_1}
	\|\dD_{\tau_k}^{(1)}\|_{\xX(\iI_{k+1})} \leq C_{m,M} \cdot \eta^{\frac{4-\eps}{5}} \Big( \|v_{m,\eps_1} - v_{m,\eps_2}\|_{\xX(0,\tau_k)} + \|v_{m,\eps_1}-v_{m,\eps_2}\|_{\xX(\iI_{k+1})} \Big). 
	\end{equation}
	Here, $\eta$ is the small increment of the $5$-th power of $\xX_2$-norm in the dissection (see \eqref{eq:stability_dissection}). Its value will be specified below. 
	
	So far all the bounds above hold on all of $\Omega$. But for $\dD_{\tau_k}^{(2)}$, the two nonlinearities have different powers for the same input $v_{m,\eps_2}$, so we need a bound for $\|v_{m,\eps}\|_{L_{t}^{\infty}L_{x}^{\infty}}$. This is where we use $\omega \in \Omega_\Lambda$. Since $(1,2)$ is dual of the Strichartz pair $(+\infty,2)$, applying the Strichartz estimates \eqref{eq:Strichartz_2}, we get
	\begin{equation*}
	\begin{split}
	\|\dD_{\tau_k}^{(2)}\|_{\xX(\iI_{k+1})} &\leq C \big\| v_{m,\eps_2} \big( |v_{m,\eps_2}|^{4-\eps_1} - |v_{m,\eps_2}|^{4-\eps_2} \big) \big\|_{L_{t}^{1}L_{x}^{2}(\iI_{k+1})}\\
	&\leq C \|v_{m,\eps_2}\|_{\xX_{1}(\iI_{k+1})} \cdot \big\|\;|v_{m,\eps_2}|^{4-\eps_1} - |v_{m,\eps_2}|^{4-\eps_2} \big\|_{L_{t}^{\infty} L_{x}^{\infty}(\iI_{k+1})}. 
	\end{split}
	\end{equation*}
	Since we are in dimension one, we have
	\begin{equation*}
	\|v_{m,\eps_j}\|_{L_{x}^{\infty}}^{2} \leq 2 \|v_{m,\eps_j}\|_{L_{x}^{2}} \|\d_x v_{m,\eps_j}\|_{L_{x}^{2}} \leq 2 \Lambda^{2}, 
	\end{equation*}
	where the first inequality follows from Newton-Leibniz for the function $(v_{m,\eps_j})^{2}$, and in the second inequality we have used the assumption that $\omega \in \Omega_{\Lambda}$. Hence, we have the pointwise bound
	\begin{equation*}
	\Big| \big(|v_{m,\eps_2}|^{4-\eps_1} - |v_{m,\eps_2}|^{4-\eps_2}\big) \1_{\Omega_\Lambda} \Big| \leq C \Lambda^{5} |\eps_1 - \eps_2|. 
	\end{equation*}
	Plugging it back to the bound for $\dD_{\tau_k}^{(2)}$ above, we get
	\begin{equation} \label{eq:cauchy_N_2}
	\|\dD_{\tau_k,2}^{(2)} \1_{\Omega_\Lambda}\|_{\xX(\iI_{k+1})} \leq C M \Lambda^{5} |\eps_1 - \eps_2| \leq C M \Lambda^{5} \eps^{*}. 
	\end{equation}
	Plugging \eqref{eq:cauchy_N_1} and \eqref{eq:cauchy_N_2} back into \eqref{eq:cauchy_deter_1}, and letting $\eta$ be small enough so that we can merge the term $\|v_{m,\eps_1} - v_{m,\eps_2}\|_{\xX_2(\iI_{k+1})}$ arising from $\|\dD_{\tau_k}^{(1)}\|_{\xX(\iI_{k+1})}$ to the left hand side, we get
	\begin{equation*}
	\|(v_{m,\eps_1} - v_{m,\eps_2}) \1_{\Omega_\Lambda}\|_{\xX(\iI_{k+1})} \leq C_{m,M} \Big( \|(v_{m,\eps_1} - v_{m,\eps_2}) \1_{\Omega_\Lambda}\|_{\xX(0,\tau_k)} +  \Lambda^{5} \eps^{*} + \mM_{a,b}^{*} \Big). 
	\end{equation*}
	Again, iterating this bound over the intervals $\iI_1, \dots, \iI_K$ and adding $\|(v_{m,\eps_1} - v_{m,\eps_2}) \1_{\Omega_{\Lambda}}\|_{\xX(0,a)}$ to both sides, we get
	\begin{equation*}
	\|(v_{m,\eps_1} - v_{m,\eps_2}) \1_{\Omega_\Lambda}\|_{\xX(0,b)} \leq C_{m,M} \Big( \|(v_{m,\eps_1} - v_{m,\eps_2}) \1_{\Omega_\Lambda}\|_{\xX(0,a)} + \Lambda^{5} \eps^{*} + \mM_{a,b}^{*} \Big). 
	\end{equation*}
	Taking $L_{\omega}^{\rho_0}$-norm on both sides and using the bound \eqref{eq:cauchy_stochastic}, we get
	\begin{equation} \label{eq:cauchy_average_2}
	\begin{split}
	\|(v_{m,\eps_1} - v_{m,\eps_2}) \1_{\Omega_{\Lambda}}\|_{L_{\omega}^{\rho_0}\xX(0,b)} \leq &C_{m,M} \Big( \|v_{m,\eps_1} - v_{m,\eps_2} \|_{L_{\omega}^{\rho_0} \xX(0,a)} + \Lambda^{5} \eps^{*}\\
	&+ h^{\frac{3}{10}} \|v_{m,\eps_1} - v_{m,\eps_2} \|_{L_{\omega}^{\rho_0} \xX(a,b)} \Big), 
	\end{split}
	\end{equation}
	where the constant $C_{m,M}$ is deterministic and depends on $m$ and $M$ only, and we have also relaxed the right hand side by removing $\1_{\Omega_\Lambda}$ in the first term. 
	
	Now, adding \eqref{eq:cauchy_average_1} and \eqref{eq:cauchy_average_2} together, we get a bound for $\|(v_{m,\eps_1} - v_{m,\eps_2})\|_{L_{\omega}^{\rho_0}\xX(0,b)}$. Hence, if we choose $h$ sufficiently small (depending on $m$ and $M$ this time), we can again absorb the term $\|v_{m,\eps_1} - v_{m,\eps_2}\|_{L_{\omega}^{\rho_0} \xX(a,b)}$ to the left hand side of the bound. This gives us
	\begin{equation} \label{eq:cauchy_local}
	\|v_{m,\eps_1} - v_{m,\eps_2}\|_{L_{\omega}^{\rho_0}\xX(0,b)} \leq C^{*} \Big( \|v_{m,\eps_1} - v_{m,\eps_2}\|_{L_{\omega}^{\rho_0}\xX(0,a)} + \Lambda^{5} \eps^{*} + \Lambda^{-\frac{1}{2}} \Big), 
	\end{equation}
	where $C^{*} = C^{*}(m,M,\tilde{M})$. Now, for every $\delta>0$, we choose $\kappa$ small enough so that $C^{*} \kappa < \frac{\delta}{3}$. We then choose$\Lambda$ large enough so that $C^{*} \Lambda^{-\frac{1}{2}} < \frac{\delta}{3}$, and finally $\eps^{*}$ small so that $C^{*} \Lambda^{5} \eps^{*} < \frac{\delta}{3}$. This completes the proof of the lemma. 
\end{proof}

\begin{proof} [Proof of Proposition~\ref{pr:convergence_high}]
	It suffices to show that for every $\delta>0$, there exists $\eps^{*}>0$ such that
	\begin{equation} \label{eq:high_aim}
	\|v_{m,\eps_1} - v_{m,\eps_2}\|_{L_{\omega}^{\rho_0}\xX(0,1)} < \delta
	\end{equation}
	whenever $\eps_1, \eps_2 < \eps^{*}$. The proof is a backward induction argument using Lemma~\ref{le:Cauchy_high_local}. Fix arbitrary $\delta>0$, and let $\kappa^{(0)} = \delta$. By Lemma~\ref{le:Cauchy_high_local}, there exists finite positive sequence $\{\eps^{(n)}, \kappa^{(n)}\}_{n=1}^{N(h)}$ such that
	\begin{equation*}
	\eps_1, \eps_2 < \eps^{(n)} \quad \text{and} \quad \|v_{m,\eps_1} - v_{m,\eps_2}\|_{L_{\omega}^{\rho_0}\xX(0,1-nh)} < \kappa^{(n)}
	\end{equation*}
	implies $\|v_{m,\eps_1} - v_{m,\eps_2}\|_{L_{\omega}^{\rho_0}\xX(0,1-(n-1)h)} < \kappa^{(n-1)}$. Here, $h$ is the same as in Lemma~\ref{le:Cauchy_high_local} and $N(h) = 1 + \lfl \frac{1}{h} \rfl$. We then take
	\begin{equation*}
	\eps^{*} = \min_{n \leq N(h)} \eps^{(n)}. 
	\end{equation*}
	The proof is complete by noting that $v_{m,\eps_1}$ and $v_{m,\eps_2}$ have the same initial data so they start with zero difference. 
\end{proof}

\subsection{Proof of Proposition~\ref{pr:converge_e}}

We are now ready to prove the convergence in $\eps$. Let $u_{m,\eps_1}$ and $u_{m,\eps_2}$ denote the solutions to \eqref{eq:trun_sub} with common initial data $u_0$ and nonlinearities $\nN^{\eps_1}$ and $\nN^{\eps_2}$ respectively. We first need the following lemma to perturb the initial data to $H_{x}^{1}$. 

\begin{lem} \label{le:initial_smoothen}
	For every $u_0 \in L_{\omega}^{\infty}L_{x}^{2}$ and every $\kappa > 0$, there exists $v_0 \in L_{\omega}^{\infty}H_{x}^{1}$ such that $\|v_0\|_{L_x^2} \leq \|u_0\|_{L_x^2}$ almost surely, and
	\begin{equation*}
	\|u_0 - v_0\|_{L_{\omega}^{\rho_0}L_{x}^{2}} < \kappa. 
	\end{equation*}
\end{lem}
\begin{proof}
	We fix $u_0 \in L_{\omega}^{\infty}L_{x}^{2}$ and $\kappa > 0$ arbitrary. Let $\varphi$ be a mollifier on $\RR$, and $\varphi_{\delta} = \delta^{-1} \varphi(\cdot / \delta)$. For every $\omega \in \Omega$, let $u_{0,\delta}(\omega) = u_0(\omega) * \varphi_{\delta}$. By Young's inequality, we have
	\begin{equation} \label{eq:initial_convolution}
	\|u_{0,\delta}(\omega)\|_{L_{x}^{2}} \leq \|u_{0}(\omega)\|_{L_{x}^{2}}, \quad \|u_{0,\delta}(\omega)\|_{\dot{H}_{x}^{1}} \leq \delta^{-1} \|\varphi'\|_{L_{x}^{1}} \|u_{0}(\omega)\|_{L_{x}^{2}}
	\end{equation}
	for almost every $\omega$ and every $\delta>0$. In addition, 
	\begin{equation*}
	\|u_{0,\delta}(\omega) - u_{0}(\omega)\|_{L_{x}^{2}} \rightarrow 0
	\end{equation*}
	as $\delta \rightarrow 0$ for almost every $\omega$. Hence, by Egorov's theorem, there exists $\Omega' \subset \Omega$ with
	\begin{equation*}
	\Pr(\Omega') \leq \Big( \frac{\kappa}{4 \|u_0\|_{L_{\omega}^{\infty}L_{x}^{2}}} \Big)^{\rho_0}
	\end{equation*}
	such that on $\Omega \setminus \Omega'$, we have
	\begin{equation*}
	\sup_{\omega \in \Omega \setminus \Omega'} \|u_{0,\delta}(\omega) - u_{0}(\omega)\|_{L_{x}^{2}} \rightarrow 0
	\end{equation*}
	as $\delta \rightarrow 0$. Thus, we can choose $\delta = \delta^{*}$ small enough so that
	\begin{equation*}
	\sup_{\omega \in \Omega \setminus \Omega'} \|u_{0,\delta^{*}}(\omega) - u_{0}(\omega)\|_{L_{x}^{2}} < \frac{\kappa}{2}. 
	\end{equation*}
	Let $v_{0}(\omega) = u_{0,\delta^{*}}(\omega)$. It then follows immediately from \eqref{eq:initial_convolution} that $v_0 \in L_{\omega}^{\infty} H_{x}^{1}$ and $\|v_0\|_{L_{x}^{2}} \leq \|u_0\|_{L_{x}^{2}}$ almost surely. As for $\|u_0-v_0\|_{L_{\omega}^{\rho_0}L_{x}^{2}}$, we have
	\begin{equation*}
	\begin{split}
	\|u_0 - v_0\|_{L_{\omega}^{\rho_0} L_{x}^{2}} &\leq \|(u_0-v_0) \1_{\Omega \setminus \Omega'}\|_{L_{\omega}^{\rho_0}L_{x}^{2}} + \|(u_0-v_0) \1_{\Omega'}\|_{L_{\omega}^{\rho_0}L_{x}^{2}}\\
	&\leq \frac{\kappa}{2} + \big( \|u_0\|_{L_{\omega}^{\infty}L_{x}^{2}} + \|v_0\|_{L_{\omega}^{\infty}L_{x}^{2}} \big) \big( \Pr(\Omega')\big)^{\frac{1}{\rho_0}}\\
	&\leq \kappa. 
	\end{split}
	\end{equation*}
	This completes the proof. 
\end{proof}

\begin{proof} [Proof of Proposition~\ref{pr:converge_e}]
	We first show that, for every fixed $m$, the sequence $\{u_{m,\eps}\}_{\eps>0}$ is Cauchy in $L_{\omega}^{\rho_0} \xX(0,1)$. Fix an arbitrary $\delta>0$. For every $\kappa>0$, by Lemma~\ref{le:initial_smoothen}, we can choose $v_0 \in L_{\omega}^{\infty}H_{x}^{1}$ such that
	\begin{equation*}
	\|v_0\|_{L_{\omega}^{\infty}L_{x}^{2}} \leq \|u_0\|_{L_{\omega}^{\infty}L_{x}^{2}} \leq M, \quad \|u_0 - v_0\|_{L_{\omega}^{\rho_0}L_{x}^{2}} < \kappa. 
	\end{equation*}
	For every $\eps_1$ and $\eps_2$, we have
	\begin{equation*}
	\begin{split}
	\|u_{m,\eps_1} - u_{m,\eps_2}\|_{L_{\omega}^{\rho_0}\xX(0,1)} \leq &\|u_{m,\eps_1} - v_{m,\eps_1}\|_{L_{\omega}^{\rho_0}\xX(0,1)} + \|u_{m,\eps_2} - v_{m,\eps_2}\|_{L_{\omega}^{\rho_0}\xX(0,1)}\\
	&+ \|v_{m,\eps_1} - v_{m,\eps_2}\|_{L_{\omega}^{\rho_0} \xX(0,1)}, 
	\end{split}
	\end{equation*}
	where $v_{m,\eps_j}$ is the solution to \eqref{eq:duhamel_trun_sub} with initial data $v_0$ and nonlinearity $\nN^{\eps_j}$. By Proposition~\ref{pr:uniform_stable}, we can let $\kappa$ be sufficiently small so that the first two terms on the right hand side above are both smaller than $\frac{\delta}{3}$, uniformly in $\eps_1, \eps_2 \in (0,1)$. As for the third term, by Proposition~\ref{pr:convergence_high}, we can choose $\eps^{*}$ sufficiently small so that this term is also smaller than $\frac{\delta}{3}$ as long as $\eps_1, \eps_2 < \eps^{*}$. This shows that $\{u_{m,\eps}\}_{\eps}$ is Cauchy in $L_{\omega}^{\rho_0}\xX(0,1)$, and hence has a limit in the same space, which we denote by $u_m$. 
	
	It then remains to show that the limit $u_{m}$ satisfies the equation \eqref{eq:trun_critical}. To see this, it suffices to show that each term on the right hand side of \eqref{eq:duhamel_trun_sub} converges in $L_{\omega}^{\rho_0}\xX(0,1)$ to the corresponding term with $u_{m,\eps}$ replaced by $u_{m}$. The term with the initial data are identical for all $\eps \geq 0$. The convergence of
	\begin{equation*}
	\int_{0}^{t} \sS(t-s) u_{m,\eps}(s) {\rm d} W_s \quad \text{and} \quad \int_{0}^{t} \sS(t-s) \big( F_{\Phi} u_{m,\eps}(s) \big) {\rm d}s
	\end{equation*}
	follows immediately from the bounds in Propositions~\ref{pr:pre_stochastic} and~\ref{pr:pre_correction} and that $\|u_{m,\eps}-u_{m}\|_{L_{\omega}^{\rho_0}\xX(0,1)} \rightarrow 0$. As for the term with the nonlinearity, one needs to compare $\nN^{\eps}(u_{m,\eps})$ and $\nN(u_{m})$, where we recall $\nN(u_m) = |u_m|^{4} u_m$. Hence, we proceed as above to perturb the initial data to $v_0 \in L_{\omega}^{\infty}H_x^1$. Let $v_{m,\eps}$ denotes the solution to \eqref{eq:duhamel_trun_sub} and $v_m$ denotes its limit in $L_{\omega}^{\rho_0}\xX(0,1)$ as $\eps \rightarrow 0$. We then write
	\begin{equation*}
	\begin{split}
	&\phantom{111}\theta_{m}\big(\|u_{m,\eps}\|_{\xX_2(0,s)}^{5} \big) \nN^{\eps}(u_{m,\eps}) - \theta_{m} \big(\|u_{m}\|_{\xX_2(0,s)}^{5}\big) \nN(u_m)\\
	&= \theta_{m}\big(\|u_{m,\eps}\|_{\xX_2(0,s)}^{5}\big) \nN^{\eps}(u_{m,\eps}) - \theta_{m} \big(\|v_{m,\eps}\|_{\xX_2(0,s)}^{5}\big) \nN^{\eps}(v_{m,\eps})\\
	&+ \theta_{m} \big(\|v_{m,\eps}\|_{\xX_2(0,s)}^{5}\big) \nN^{\eps}(v_{m,\eps}) - \theta_{m} \big(\|v_{m}\|_{\xX_2(0,s)}^{5}\big) \nN(v_{m})\\
	&+ \theta_{m} \big(\|v_{m}\|_{\xX_2(0,s)}^{5}\big) \nN(v_{m}) - \theta_{m} \big(\|u_{m}\|_{\xX_2(0,s)}^{5}\big) \nN(u_m). 
	\end{split}
	\end{equation*}
	By Proposition~\ref{pr:uniform_stable}, we have
	\begin{equation*}
	\sup_{\eps \in [0,1]} \|u_{m,\eps} - v_{m,\eps}\|_{L_{\omega}^{\rho_0} \xX(0,1)} \leq \|u_0 - v_0\|_{L_{\omega}^{\rho_0}L_{x}^{2}}, 
	\end{equation*}
	where $\eps=0$ case corresponds to $u_m - v_m$. Thus, with Strichartz estimate \eqref{eq:Strichartz_2} and Lemma~\ref{le:initial_smoothen}, we can make the nonlinearity from the first and third terms above arbitrarily small by making $\|u_0 - v_0\|_{L_{\omega}^{\rho_0}L_{x}^{2}}$ small enough. As for the second term, we can get the desired bound from Lemma~\ref{le:persistence_short} with $\rho=2\rho_0$ and Proposotion~\ref{pr:persistence} with $\rho = 2\rho_0$ in exactly the same way as above. Note that the arguments here are only much simpler since we already have apriori bounds on $\|\uem - \vem\|_{L_{\omega}^{\rho_0}\xX(0,1)}$ and $\|\vem - v_m\|_{L_{\omega}^{\rho_0}\xX(0,1)}$. We can then conclude that the limit $u_{m}$ solves the truncated critical equation \eqref{eq:trun_critical}. 
\end{proof}

\section{Uniform boundedness of $u_m$ -- proof of Proposition~\ref{pr:uniform_m}}
\label{sec:uniform_m}

This section is devoted to the proof of Proposition~\ref{pr:uniform_m}. The main ingredients are a series of deterministic boundedness and stability statements, whose proof rely on the recent scattering results of mass-critical NLS by Dodson (\cite{dodson2012global, dodson2016global, dodson2467global}). We will first review these results in Section~\ref{sec:Dodson}. In Section~\ref{sec:deterministic_m}, we state and prove the uniform-in-$m$ version of the theorems by Dodson, and then use them to prove the key deterministic bound in Proposition~\ref{pr:main_deter_sta}. Finally, in Section~\ref{sec:pf_uniform_m}, we complete the proof of Proposition~\ref{pr:uniform_m} by using the bound in the previous subsection. 

Throughout this section, we fix the interval $\iI=[a,b]$ with $b-a \leq 1$, and all the bounds are uniform of the intervals with this constraint.

\subsection{Brief review of stability for the mass-critical NLS}
\label{sec:Dodson}

Let $\mu \in [0,1]$. Let $w \in \xX(\iI)$ be the solution to
\begin{equation} \label{eq:nls_w}
i \d_t w + \Delta w = \mu |w|^{4} w\;, \qquad w(a) \in L_{x}^{2}, 
\end{equation}
and $v \in \xX(\iI)$ and $e \in L_{t}^{1}L_{x}^{2}(\iI)$ such that
\begin{equation} \label{eq:nls_v}
i \d_t v + \Delta v = \mu |v|^4 v + e\;, \qquad v(a) \in L_{x}^{2}. 
\end{equation}
We call \eqref{eq:nls_w} the mass-critical NLS with parameter $\mu$. The following stability result is well known. 

\begin{prop} \label{pr:iteam_sta}
	Let $w \in \xX(\iI)$ be the solution to \eqref{eq:nls_w}. Let $v,e$ satisfy \eqref{eq:nls_v}. Then for every $M_1, M_2 > 0$, there exist $\delta$ and $C$ depending on $M_1$ and $M_2$ only such that if
	\begin{equation*}
	\|v\|_{\xX_1(\iI)} \leq M_1, \quad \|v\|_{\xX_2(\iI)} \leq M_2, \quad \|v(a)-w(a)\|_{L_{x}^{2}} + \|e\|_{L_{t}^{1}L_{x}^{2}(\iI)} \leq \delta, 
	\end{equation*}
	then we have
	\begin{equation*}
	\|v-w\|_{\xX(\iI)} \leq C \big( \|v(a)-w(a)\|_{L_x^2} + \|e\|_{L_{t}^{1}L_{x}^{2}(\iI)} \big). 
	\end{equation*}
	In particular, $\delta$ and $C$ do not depend on $\mu \in [0,1]$. 
\end{prop}

The above proposition is purely perturbative since it assumes the boundedness of $\|v\|_{\xX_2(\iI)}$. More details of the proposition can be found in \cite[Lemmas~3.9 and~3.10]{colliander2008global}. On the other hand, the well-posedness of \eqref{eq:nls_w} on the interval $\iI$ is highly nontrivial, and it is proved recently by Dodson (\cite{dodson2012global, dodson2016global, dodson2467global}). We state it (in $d=1$) below. 

\begin{thm} [Dodson]
	\label{th:Dodson}
	For every $w(a) \in L_{x}^{2}$, the equation \eqref{eq:nls_w} has a global solution $w \in \xX(\RR)$. Furthermore, we have the bound
	\begin{equation*}
	\|w\|_{\xX(\RR)} \leq D_M, 
	\end{equation*}
	where $D_M$ depends on $\|w(a)\|_{L_{x}^{2}}$ only. 
\end{thm}

The proof by Dodson used Proposition~\ref{pr:iteam_sta}. Actually, it shows that the solutions is not only bounded on the whole real line, but also scatters. More details of the scattering properties can be found in the above mentioned papers by Dodson. 

Combining Proposition~\ref{pr:iteam_sta} and Theorem~\ref{th:Dodson}, we have the following stronger stability property. 

\begin{prop} \label{pr:strong_sta}
	Let $w$ be the solution to \eqref{eq:nls_w}, and $v$, $e$ satisfy \eqref{eq:nls_v}. For every $M>0$, there exist $\delta_M, C_M>0$ depending on $M$ only such that if
	\begin{equation*}
	\|w(a)\|_{L_{x}^{2}} \leq M\;, \qquad \|v(a)-w(a)\|_{L_x^2} + \|e\|_{L_{t}^{1}L_{x}^{2}(\iI)} \leq \delta_M, 
	\end{equation*}
	then we have
	\begin{equation*}
	\|v-w\|_{\xX(\iI)} \leq C_M \big( \|v(a)-w(a)\|_{L_x^2} + \|e\|_{L_{t}^{1}L_{x}^{2}(\iI)} \big). 
	\end{equation*}
	The bound is uniform over intervals $\iI$ with $|\iI| \leq 1$ and coupling constant $\mu \in [0,1]$. 
\end{prop}

\subsection{Deterministic boundedness and stability}
\label{sec:deterministic_m}

The aim of this subsection is to prove the key deterministic bound in Proposition~\ref{pr:main_deter_sta}. This bound is written in the form that is suitable for proving uniform boundedness of $\{u_m\}$. Our proof of this bound uses some parallel statements to those in Section~\ref{sec:Dodson} but with $m<+\infty$. 

Throughout, we let $m, A, \tilde{A}>0$, and $w_m, v_m \in \xX(\iI)$ and $e \in L_{t}^{1} L_{x}^{2}(\iI)$ satisfy the equations
\begin{equation} \label{eq:nls_w_m}
i \d_t w_m + \Delta w_m = \theta_{m}\big( A + \|w_m\|_{\xX_2(a,t)}^{5} \big) |w_m|^{4} w_m\;, \quad w_m(a) \in L_{x}^{2}, 
\end{equation}
and
\begin{equation} \label{eq:nls_v_m}
i \d_t v_m + \Delta v_m = \theta_{m}\big(\tilde{A} + \|v_m\|_{\xX_2(a,t)}^{5}\big) |v_m|^4 v_m + e\;, \quad v_m(a) \in L_{x}^{2}. 
\end{equation}
Both equations hold in $\iI$. The following is an $m$-version of Theorem~\ref{th:Dodson}. 


\begin{prop} [Uniform-in-$m$ boundedness]
	\label{pr:bd_m}
	Let $w_m \in \xX(\iI)$ be the solution to \eqref{eq:nls_w_m}. Then for every $M > 0$, there exists $\tilde{D}_M > 0$ such that
	\begin{equation*}
	\|w_m\|_{\xX(\iI)} \leq \tilde{D}_M
	\end{equation*}
	whenever $\|w_m(a)\|_{L_x^2} \leq M$. The bound is uniform in $m$ and $A$. 
\end{prop}
\begin{proof}
	Fix $M>0$ arbitrary. Note that $w_m$ satisfies a free equation once its $\xX_2$-norm reaches $2m$, so it is defined on all of $[a,b]$. We only need to show that it is bounded by a constant depending on $M$ only. 
	
	If $m$ is smaller than some (possibly large) constant depending on $M$, then $w$ will satisfy the free equation after its $\xX_2$-norm reaches twice that constant. Hence, $\|w\|_{\xX_2(\iI)}$ will be bounded by a constant depending on $M$ only, giving the desired control. Thus, it suffices to consider the case when $m$ is large. 
	
	We first re-write the equation for $w_m$ as
	\begin{equation*}
	i \d_t w_m + \Delta w_m = \theta_m(A) |w_m|^4 w_m + e, 
	\end{equation*}
	where
	\begin{equation*}
	e(t) = \big( \theta_m(A + \|w_m\|_{\xX_2(a,t)}^{5}) - \theta_m(A) \big) |w_m(t)|^{4} w_m(t). 
	\end{equation*}
	Let $w$ denote the solution to \eqref{eq:nls_w} with $\mu = \theta_m(A)$ and $w(a) = w_{m}(a)$. If $\|e\|_{L_{t}^{1}L_{x}^{2}(a,b)}$ is small, then we can use Proposition~\ref{pr:strong_sta} to control the difference between $w_m$ and $w$, and the desired bound will follow. But $e$ itself depends on $w_m$, so we need to use bootstrap arguments to get the desired control. 
	
	For every $a \leq t \leq r \leq b$, we have the pointwise bound
	\begin{equation*}
	|e(t)| \leq \frac{C_0}{m} \|w_m\|_{\xX_2(a,t)}^{5} |w_m(t)|^{5} \leq \frac{C_0}{m} \|w_m\|_{\xX_2(a,r)}^{5} |w_m(t)|^{5}. 
	\end{equation*}
	By H\"older's inequality, we get
	\begin{equation} \label{eq:bd_m_error}
	\|e\|_{L_{t}^{1}L_{x}^{2}(a,r)} \leq \frac{C_0}{m} \|w_m\|_{\xX_2(a,r)}^{10}. 
	\end{equation}
	Let $D_M$, $\delta_M$ and $C_M$ be the constants as in Theorem~\ref{th:Dodson} and Proposition~\ref{pr:strong_sta}. Let $m$ be sufficiently large so that
	\begin{equation} \label{eq:choice_m}
	\frac{C_0}{m} \Big( 2(D_M + C_M \delta_M) \Big)^{10} \leq \delta_M, 
	\end{equation}
	where $C_0$ is the same constant as in \eqref{eq:bd_m_error}. This choice of $m$ depends on $M$ only. We claim that $\|w_m\|_{\xX(a,r)}$ will never exceed $2(D_M + C_M \delta_M)$ on $[a,b]$. To see this, we first note that $\|w_m\|_{\xX(a,a)} = \|w_m(a)\|_{L_{x}^{2}} \leq M \leq D_M$. Let
	\begin{equation*}
	\tau := \inf \big\{ r>a: \|w_m\|_{\xX(a,r)} = 2(D_M+C_M\delta_M) \big\}. 
	\end{equation*}
	Then, the choice of $m$ in \eqref{eq:choice_m} guarantees that
	\begin{equation*}
	\|e\|_{L_{t}^{1}L_{x}^{2}(a,\tau)} \leq \delta_M. 
	\end{equation*}
	By Theorem~\ref{th:Dodson} and Proposition~\ref{pr:strong_sta}, we have
	\begin{equation*}
	\|w_m\|_{\xX(a,\tau)} \leq \|w\|_{\xX(a,\tau)} + C_M \delta_M \leq D_M + C_M \delta_M. 
	\end{equation*}
	This shows that $\|w\|_{\xX_2(a,r)} \leq 2(D_M + C_M \delta_M)$ for all $r \in [a,b]$, and in particular it holds for $r=b$. This completes the proof. 
\end{proof}

With the uniform boundedness, we have the following stability statement. 

\begin{prop} [Uniform-in-$m$ stability]
	\label{pr:strong_sta_m}
	Let $w_m, v_m \in \xX(\iI)$ and $e \in L_{t}^{1}L_{x}^{2}(\iI)$ satisfy \eqref{eq:nls_w_m} and \eqref{eq:nls_v_m}. Then for every $M>0$, there exist constants $\tilde{\delta}_M, \tilde{C}_M>0$ such that if
	\begin{equation*}
	\|v_m(a)\|_{L_{x}^{2}} \leq M\;, \quad \|v_m(a)-w_m(a)\|_{L_x^2} + \|e\|_{L_{t}^{1}L_{x}^{2}(\iI)} + |A-\tilde{A}| \leq \tilde{\delta}_M, 
	\end{equation*}
	then we have
	\begin{equation*}
	\|v_m - w_m\|_{\xX(\iI)} \leq \tilde{C}_M \big( \|v_m(a)-w_m(a)\|_{L_x^2} + \|e\|_{L_{t}^{1}L_{x}^{2}(\iI)} + |A-\tilde{A}| \big). 
	\end{equation*}
	The constants $\tilde{\delta}_M$ and $\tilde{C}_M$ are independent of $m$, $A$ and $\tilde{A}$. 
\end{prop}
\begin{proof}
	Let $\tilde{D}_M$ be the constant in Proposition~\ref{pr:bd_m}. Let $\eta>0$ be a small number whose value will be specified later. Let $\{\tau_k\}_{k=0}^{K}$ be a dissection of the interval $[a,b]$ given by $\tau_0 = a$, and
	\begin{equation*}
	\tau_{k+1} = b \wedge \inf\Big\{ r>\tau_k: \|w_m\|_{\xX_2(\tau_k,r)}^{5} = \eta \Big\}. 
	\end{equation*}
	The total number of intervals in this dissection is then at most
	\begin{equation} \label{eq:sta_m_interval}
	K \leq 1 + \frac{\tilde{D}_M^5}{\eta}. 
	\end{equation}
	Let $\iI_{k+1} = [\tau_k, \tau_{k+1}]$. For every $k$ and every $t \in \iI_{k+1}$, we have
	\begin{equation} \label{eq:sta_m_express}
	v_{m}(t) - w_{m}(t) = e^{i(t-\tau_k)\Delta} \big( v_{m}(\tau_k) - w_{m}(\tau_k) \big) - \int_{\tau_k}^{t} \sS(t-s) e(s) {\rm d}s + \dD_{\tau_k}(t), 
	\end{equation}
	where
	\begin{equation*}
	\begin{split}
	\dD_{\tau_k}(t) = -i \int_{\tau_k}^{t} &\sS(t-s) \Big( \theta_{m}\big(\tilde{A}+\|v_m\|_{\xX_2(a,s)}^{5}\big) \nN\big(v_m(s)\big)\\
	&- \theta_{m}\big(A+\|w_m\|_{\xX_2(a,s)}^{5}\big) \nN\big(w_m(s)\big)\Big) {\rm d}s. 
	\end{split}
	\end{equation*}
	For the first two terms on the right hand side of \eqref{eq:sta_m_express}, we have
	\begin{equation} \label{eq:sta_m_easy}
	\begin{split}
	\Big\| e^{i(t-\tau_k)\Delta}\big(v_m(\tau_k) - w_m(\tau_k)\big) \Big\|_{\xX(\iI_{k+1})} &\leq C \|v_m - w_m\|_{\xX(a,\tau_k)}, \\
	\Big\| \int_{\tau_k}^{t} \sS(t-s) e(s) {\rm d}s \Big\|_{\xX(\iI_{k+1})} &\leq C \|e\|_{L_{t}^{1}L_{x}^{2}(a,b)}, 
	\end{split}
	\end{equation}
	where both constants $C$ are universal. Now, for every $r \in \iI_{k+1}$, taking $\xX(\tau_k,r)$-norm on both sides of \eqref{eq:sta_m_express}, using the bounds \eqref{eq:sta_m_easy}, and then adding $\|v_m - w_m\|_{\xX(a,\tau_k)}$ to both sides, we get
	\begin{equation} \label{eq:sta_m_intermediate}
	\|v_m - w_m\|_{\xX(a,r)} \leq C \big( \|v_m - w_m\|_{\xX(a,\tau_k)} + \|e\|_{L_{t}^{1}L_{x}^{1}(a,b)} \big) + \|\dD_{\tau_k}\|_{\xX(\tau_k,r)}. 
	\end{equation}
	Here, the constant $C$ is universal, and in particular does not depend on $r \in \iI_{k+1}$. 
	
	It remains to control $\|\dD_{\tau_k}\|_{\xX(\tau_k,r)}$. The integrand in $\dD_{\tau_k}$ satisfies the pointwise bound
	\begin{equation*}
	\begin{split}
	&\phantom{111}\Big| \theta_{m}\big(\tilde{A}+\|v_m\|_{\xX_2(a,s)}^{5}\big) \nN\big(v_m(s)\big) - \theta_{m}\big(A+\|w_m\|_{\xX_2(a,s)}^{5}\big) \nN\big(w_m(s)\big)\Big|\\
	&\leq C |w_m(s)|^{5} \Big( |\tilde{A}-A| + \tilde{D}_{M}^{4} \|v_m - w_m\|_{\xX_2(a,r)} + \|v_m-w_m\|_{\xX_2(a,r)}^{5} \Big)\\
	&+ C |v_m(s)-w_m(s)| \big( |w_m(s)|^{4} + |v_{m}(s) - w_{m}(s)|^{4} \big)
	\end{split}
	\end{equation*}
	for every $\tau_k \leq s \leq r \leq \tau_{k+1}$. Hence, by Strichartz estimates \eqref{eq:pre_nonlinear_critical}, we have
	\begin{equation*} 
	\begin{split}
	\|\dD_{\tau_k}\|_{\xX(\tau_k,r)} &\leq C \|w_m\|_{\xX_2(\tau_k,r)}^{5} \Big( |A-\tilde{A}| + \tilde{D}_M^{4} \|v_m-w_m\|_{\xX_2(a,r)} + \|v_m-w_m\|_{\xX_2(a,r)}^{5} \Big)\\
	&+ C \|v-w\|_{\xX_2(\tau_k,r)} \Big( \|w_m\|_{\xX_2(\tau_k,r)}^{4} + \|v_m -w_m\|_{\xX_2(\tau_k,r)}^{4} \Big)\\
	\end{split}
	\end{equation*}
	Since $\|w_m\|_{\xX_2(\tau_k,r)}^{5} \leq \eta$ for $r \in \iI_{k+1}$, we get
	\begin{equation} \label{eq:sta_m_nonlinear}
	\|\dD_{\tau_k}\|_{\xX(\tau_k,r)} \leq C \eta \Big( |\tilde{A}-A| + (\tilde{D}_{M}^4 \eta^5 + \eta^4) \|v_m-w_m\|_{\xX_2(a,r)} + \|v_m - w_m\|_{\xX_2(a,r)}^{5} \Big). 
	\end{equation}
	Substituting the bound \eqref{eq:sta_m_nonlinear} into \eqref{eq:sta_m_intermediate} and choosing $\eta$ sufficiently small so that
	\begin{equation} \label{eq:sta_m_eta}
	\eta < 1\; \quad \text{and} \quad C (\tilde{D}_{M}^{4}\eta^5 + \eta^4) \leq \frac{1}{2}, 
	\end{equation}
	we can absorb the term $\|v_m-w_m\|_{\xX_2(a,r)}$ into the right hand side of \eqref{eq:sta_m_intermediate}, and obtain
	\begin{equation} \label{eq:sta_m_bootstrap}
	\|v_m-w_m\|_{\xX(a,r)} \leq C_{0} \Big( \|v_m - w_m\|_{\xX(a,\tau_k)} + \|e\|_{L_{t}^{1}L_{x}^{2}(a,b)} + |\tilde{A}-A| + \|v_m-w_m\|_{\xX(a,r)}^{5} \Big). 
	\end{equation}
	This bound holds for all $r \in \iI_{k+1}$ with a universal $C_0$ (not depending on $M$). Also note that the choice of $\eta$ depends on $\tilde{D}_M$ (and hence $M$) only. 
	
	Now, for every $k = 0, \dots, K$, let
	\begin{equation*}
	\delta_{k} = \|v_m-w_m\|_{\xX(a,\tau_k)} + \|e\|_{L_{t}^{1}L_{x}^{2}(\iI)} + |\tilde{A} - A|. 
	\end{equation*}
	According to \eqref{eq:sta_m_bootstrap} and the standard continuity argument, there exists a universal $\delta^* > 0$ such that if $\delta_k < \delta^*$, then
	\begin{equation*}
	\|v_m-w_m\|_{\xX(a,\tau_{k+1})} \leq 2 C_0 \delta_k, 
	\end{equation*}
	and consequently
	\begin{equation} \label{eq:iterate_diff}
	\delta_{k+1} \leq (2C_0 + 1) \delta_k. 
	\end{equation}
	By \eqref{eq:sta_m_interval}, we see if the initial difference $\delta_0$ is small enough such that
	\begin{equation*}
	\delta_0 < \frac{\delta^*}{(2C_0+1)^{1 + \tilde{D}_{M}^5/\eta}}, 
	\end{equation*}
	then we can iterate \eqref{eq:iterate_diff} up to $K$ so that
	\begin{equation*}
	\delta_K \leq (2C_0+1)^{K} \delta_0. 
	\end{equation*}
	This completes the proof of the proposition by the definition of $\delta_0$ and by taking
	\begin{equation*}
	\tilde{\delta}_M = \frac{\delta^*}{(2C_0+1)^{1 + \tilde{D}_{M}^5/\eta}}\;, \qquad \tilde{C}_M = (2C_0+1)^{1+\tilde{D}_{M}^5/\eta}, 
	\end{equation*}
	where $C_0$ is the same as in \eqref{eq:sta_m_bootstrap}, and $\eta = \eta_M$ is chosen according to \eqref{eq:sta_m_eta}. 
\end{proof}

We are now ready to prove our main deterministic bound. 

\begin{prop} \label{pr:main_deter_sta}
	Recall $\nN(u) = |u|^4 u$. Suppose $u_m, g_m \in \xX(\iI)$ satisfy $g_m(a)=0$, and
	\begin{equation*}
	u_m(t) = e^{i(t-a)\Delta}u_m(a) - i \int_{a}^{t} \sS(t-s) \Big( \theta_m\big( A + \|u\|_{\xX_2(a,s)}^{5}\big) \nN\big(u_m(s)\big) \Big) {\rm d}s + g_m(t)
	\end{equation*}
	on $\iI=[a,b]$. Then for every $M>0$, there exist $\eta_M, B_M>0$ such that whenever
	\begin{equation*}
	\|u_m\|_{\xX_1(\iI)} \leq M\;, \qquad \|g_m\|_{\xX_2(\iI)} \leq \eta_M, 
	\end{equation*}
	we have
	\begin{equation*}
	\|u_m\|_{\xX_2(\iI)} \leq B_M. 
	\end{equation*}
	The bound is uniform in $m$ and $A$. 
\end{prop}
\begin{proof}
	Let $v_m = u_m - g_m$. It suffices to bound $\|v_m\|_{\xX_2(\iI)}$. Since $g_m(a)=0$, we have
	\begin{equation*}
	\begin{split}
	v_m(t) = &e^{i(t-a)\Delta} v_m(a) - i \int_{a}^{t} \sS(t-s) \Big( \theta_{m}\big(A+ \|v_m\|_{\xX_2(a,s)}^{5}\big) \nN\big(v_m(s) \big) \Big) {\rm d}s\\
	&- i \int_{a}^{t} \sS(t-s) e(s) {\rm d}s, 
	\end{split}
	\end{equation*}
	where
	\begin{equation*}
	e(s) = \theta_{m} \big( A+\|v_m + g_m\|_{\xX_2(a,s)}^{5}\big) \nN\big(v_m(s) + g_m(s)\big) - \theta_{m} \big( A + \|v_m\|_{\xX_2(a,s)}^{5} \big) \nN\big(v_m(s)\big). 
	\end{equation*}
	We can write the above identity for $v_m$ in the differential form as
	\begin{equation*}
	i \d_t v_m + \Delta v_m = \theta_{m}\big(A+\|v_m\|_{\xX_2(a,t)}^{5}\big) \nN(v_m) + e. 
	\end{equation*}
	In view of Propositions~\ref{pr:bd_m} and~\ref{pr:strong_sta_m}, it suffices to control $\|e\|_{L_{t}^{1}L_{x}^{2}(\iI)}$. But the quantity itself depends on $v_m$, so we use the bootstrap argument similar to that in Proposition~\ref{pr:bd_m}. 
	
	Let $\eta>0$ be specified later. If $\|g_m\|_{\xX_2(a,b)} \leq \eta$, then we have the pointwise bound
	\begin{equation*}
	|e(s)| \leq C \Big(|g_m(s)| \big( |v_m(s)|^{4} + |g_m(s)|^{4} \big) + \eta |v_m(s)|^{5}  \big( \|v_m\|_{\xX_2(a,r)}^{4} + \eta^4 \big) \Big)
	\end{equation*}
	for every $a \leq s \leq r \leq b$, and $C>0$ is universal. Hence, using H\"older's inequality and again $\|g_m\|_{\xX_2(a,b)} \leq \eta$, we get
	\begin{equation} \label{eq:strong_sta_m_error}
	\|e\|_{L_{t}^{1}L_{x}^{2}(a,r)} \leq C_0 \eta \Big( \eta^{4} + \|v_m\|_{\xX_2(a,r)}^{4} + \eta^4 \|v_m\|_{\xX_2(a,r)}^{5} + \|v_m\|_{\xX_2(a,r)}^{9} \Big)
	\end{equation}
	for every $r \in [a,b]$. Let $\tilde{D}_M$, $\tilde{\delta}_M$ and $\tilde{C}_M$ be the same as in Propositions~\ref{pr:bd_m} and~\ref{pr:strong_sta_m}. Let $\eta = \eta_M$ be sufficiently small so that
	\begin{equation} \label{eq:strong_sta_m_eta}
	C_0 \eta \Big( \eta^4 + (2 \tilde{D}_M)^{4} + \eta^4 (2 \tilde{D}_M)^{5} + (2 \tilde{D}_M)^{9} \Big) \leq \tilde{\delta}_M, 
	\end{equation}
	where $C_0$ is the same as in \eqref{eq:strong_sta_m_error}. We want to show $\|v_m\|_{\xX_2(a,b)} \leq 2(\tilde{D}_M + \tilde{C}_M \tilde{\delta}_M)$ if $\|g_m\|_{\xX_2(\iI)} \leq \eta_M$. Note that $\|v_m\|_{\xX_2(a,a)}=0$. Let
	\begin{equation*}
	\tau = \inf\big\{ r>a: \|v_m\|_{\xX_2(a,r)} = 2(\tilde{D}_M + \tilde{C}_M \tilde{D}_M)\big\}. 
	\end{equation*}
	Hence, by Propositions~\ref{pr:bd_m} and~\ref{pr:strong_sta_m}, the bound \eqref{eq:strong_sta_m_error}, and the choice of $\eta$ in \eqref{eq:strong_sta_m_eta}, we have
	\begin{equation*}
	\|v_m\|_{\xX_2(a,\tau)} \leq \tilde{D}_M + \tilde{C}_M \|e\|_{L_{t}^{1}L_{x}^{2}(a,\tau)} \leq \tilde{D}_M + \tilde{C}_M \tilde{\delta}_M. 
	\end{equation*}
	This shows that $\|v_m\|_{\xX_2(a,r)}$ will never exceed $2(\tilde{D}_M + \tilde{C}_M \tilde{\delta}_M)$ in the interval $[a,b]$. Hence, 
	\begin{equation*}
	\|u_m\|_{\xX_2(\iI)} \leq \|v_m\|_{\xX_2(\iI)} + \|g_m\|_{\xX_2(\iI)} \leq 2 \big(\tilde{D}_M + \tilde{C}_M \tilde{\delta}_M\big) + \eta_M =: B_M, 
	\end{equation*}
	where $\eta_M$ is as chosen in \eqref{eq:strong_sta_m_eta}. The proof is thus complete. 
\end{proof}

\subsection{Proof of Proposition~\ref{pr:uniform_m}}
\label{sec:pf_uniform_m}

We are now ready to prove the uniform boundedness of $\{u_m\}$. We first note that Proposition~\ref{pr:converge_e} holds for every $\rho \geq 5$ and not only $\rho_0$. So we fix $\rho \geq 5$ and $M>0$ arbitrary, and let $u_0$ be independent with $W$ and that $\|u_0\|_{L_{\omega}^{\infty}L_{x}^{2}} \leq M$. Let $u_m \in L_{\omega}^{\rho}\xX(0,1)$ be the solution as in \eqref{eq:trun_critical}. We want to get a bound of $\|u_m\|_{L_{\omega}^{\rho}\xX(0,1)}$ that depends on $\rho$ and $M$ only. 

Let $h>0$ be sufficiently small whose value, depending on $M$ only, will be specified later. Let $[a,b]$ be an arbitrary subinterval of $[0,1]$ with $b-a \leq h$. We first control the $\xX_2$-norm of $u_m$ on $[a,b]$. Let
\begin{equation} \label{eq:uniform_m_martingale}
\mM^*(t) := \sup_{0 \leq r_1 \leq r_2 \leq t} \Big\| \int_{r_1}^{r_2} \sS(t-s) u_{m}(s) {\rm d} W_s \Big\|_{L_{x}^{10}}. 
\end{equation}
Note that the supremum is taken over times in $[0,t]$ but not $[a,t]$. This is because we do not need to control $\mM^*$ in terms of $h$ any more, and ranging from $0$ will be convenient for us later. By Proposition~\ref{pr:pre_stochastic}, we have
\begin{equation} \label{eq:uniform_m_martingale_bd}
\|\mM^*\|_{L_{\omega}^{\rho}L_{t}^{5}(a,b)} \leq \|\mM^*\|_{L_{\omega}^{\rho}L_{t}^5(0,1)} \leq C_{\rho} M. 
\end{equation}
In particular, $\|\mM^{*}\|_{L_{t}^{5}(a,b)}$ is almost surely finite. Hence, we can choose a random dissection $\{\tau_k\}_{k=0}^{K}$ of the interval $[a,b]$ as follows. Let $\tau_0 = a$, and define $\tau_k$ recursively by
\begin{equation*}
\tau_{k+1} = b \wedge \inf \Big\{ r>\tau_k: \int_{\tau_k}^{r} |\mM^{*}(t)|^{5} {\rm d}t = \big(\frac{\eta_M}{2}\big)^5 \Big\}, 
\end{equation*}
where $\eta_M$ is the same as in Proposition~\ref{pr:main_deter_sta}. The total number of intervals is at most
\begin{equation*}
K \leq 1 + \bigg(\frac{2\|\mM^*\|_{L_{t}^{5}(a,b)}}{\eta_M}\bigg)^{5} \leq 1 + C_{M} \|\mM^*\|_{L_{t}^{5}(0,1)}^{5}, 
\end{equation*}
where we have enlarged the range of interval to $[0,1]$, and $C_M$ is a constant that depends on $M$ only. 

Similar as before, let $\iI_{k+1} = [\tau_k, \tau_{k+1}]$. For $t \in \iI_{k+1}$, let
\begin{equation*}
g_{m}(t) = -i \int_{\tau_k}^{t} \sS(t-s) u_{m}(s) {\rm d} W_s - \frac{1}{2} \int_{\tau_k}^{t} \sS(t-s) \big( F_{\Phi} u_{m}(s) \big) {\rm d}s, 
\end{equation*}
where we omit the dependence of $g_m$ on $k$ in order to be consistent with the notation in Proposition~\ref{pr:strong_sta_m}. The choice of the dissection above ensures that
\begin{equation*}
\Big\| \int_{\tau_k}^{t} \sS(t-s) u_{m}(s) {\rm d}W_s \Big\|_{\xX_2(\iI_{k+1})} \leq \|\mM^*\|_{L_{t}^{5}(\iI_{k+1})} \leq \frac{\eta_M}{2}. 
\end{equation*}
Also, by Proposition~\ref{pr:pre_correction}, we can choose $h$ sufficiently small (but depending on $M$ only) such that
\begin{equation} \label{eq:choice_h}
\frac{1}{2} \Big\| \int_{\tau_k}^{t} \sS(t-s) \big( F_{\Phi} u_{m}(s) \big) {\rm d}s \Big\|_{\xX_2(\iI_{k+1})} \leq C M h^{\frac{4}{5}} \leq \frac{\eta_M}{2}. 
\end{equation}
The above two bounds together imply
\begin{equation*}
\|g_{m}\|_{\xX_2(\iI_{k+1})} \leq \eta_M. 
\end{equation*}
Hence, the assumption of Proposition~\ref{pr:main_deter_sta} is satisfied on the interval $\iI_{k+1}$, and we have
\begin{equation*}
\|u_m\|_{\xX_2(\iI_{k+1})}^{5} \leq B_M^5, 
\end{equation*}
where $B_M$ is also the same as in Proposition~\ref{pr:strong_sta_m}. This is true for all $k$. Hence, summing over $k$ from $0$ to $K-1$ gives
\begin{equation*}
\|u_m\|_{\xX_2(a,b)}^{5} \leq B_{M}^{5} K \leq B_{M}^{5} \big( 1 + C_M \|\mM^*\|_{L_{t}^{5}(0,1)} \big). 
\end{equation*}
This bound on $\|u_m\|_{\xX_2(a,b)}^{5}$ is uniform in the interval $[a,b]$ with $b-a<h$, so $\|u_m\|_{\xX_2(0,1)}^{5}$ is bounded by the right hand side above multiplied by $1 + \frac{1}{h}$. Since the choice of $h$ in \eqref{eq:choice_h} depends on $M$ only, we conclude that
\begin{equation*}
\|u_m\|_{\xX_2(0,1)} \lesssim_M 1 + \|\mM^*\|_{L_{t}^{5}(0,1)},  
\end{equation*}
where the proportionality constant depends on $M$ only. The proof is complete by taking $L_{\omega}^{\rho}$-norm on both sides and applying \eqref{eq:uniform_m_martingale_bd}.

\section{Removing the truncation -- proof of Theorem~\ref{th:main}}
\label{sec:converge_m}

In this section, we prove Theorem~\ref{th:main}. We will show that the sequence of solutions $\{u_{m}\}$ is Cauchy in $L_{\omega}^{\rho_{0}}\xX(0,1)$, and that the limit $u$ satisfies the corresponding Duhamel's formula for equation \ref{eq:main_eq}. The removal of the truncation $m$ relies crucially on the uniform bound in Proposition \ref{pr:uniform_m}. 

For each $m$ and $\eps$, we let
\begin{equation} \label{eq:stopping_time}
\begin{split}
\tau_{m} &= 1 \wedge \inf \big\{ t \leq 1: \|u_{m}\|_{\xX_{2}(0,t)}^{5} \geq m-1 \big\},\\
\tau_{m,\eps} &= 1 \wedge \inf \big\{ t \leq 1: \|u_{m,\eps}\|_{\xX_{2}(0,t)}^{5} \geq m \big\}. 
\end{split}
\end{equation}
Note that $\tau_m$ is different from $\tau_{m,0}$. We make $\tau_m$ the stopping time when hitting $m-1$ instead of $m$ to simplify the arguments in Lemma~\ref{le:st} below, while $\tau_m,\eps$ (including $\eps=0$) is still the time of hitting $m$. 

It has been shown in \cite[Lemma 4.1]{BD} that for every $\eps>0$ and every $m$, $\tau_{m,\eps} \leq \tau_{m+1,\eps}$ almost surely, and 
\begin{equation} \label{eq:st_dd}
u_{m,\eps} = u_{m+1,\eps} \quad \text{in}\; \xX(0,\tau_{m,\eps})
\end{equation}
almost surely. Note that the null set excluded may be different when $\eps$ changes, and it does not exclude the possibility that the set of $\omega \in \Omega$ for which \eqref{eq:st_dd} is true for all $\eps$ has probability $0$! Nevertheless, we have a similar statement for the limit $u_{m}$. 

\begin{lem} \label{le:st}
	For every $m$, we have $u_{m} = u_{m+1}$ in $\xX(0,\tau_{m} \wedge \tau_{m+1})$ almost surely. 
\end{lem}
\begin{proof}
	Fix $m$ arbitrary. Since $u_{k,\eps} \rightarrow u_{k}$ in $L_{\omega}^{\rho_0}\xX(0,1)$ for $k=m,m+1$, there exists $\Omega' \subset \Omega$ with full measure and a sequence $\eps_n \rightarrow 0$ such that $\|u_{k,\eps_n} - u_k\|_{\xX(0,1)} \rightarrow 0$ on $\Omega'$. We now only consider $\omega \in \Omega'$. Write
	\begin{equation*}
	\|u_{m} - u_{m+1}\| \leq \|u_{m} - u_{m,\eps_n}\| + \|u_{m,\eps_n}-u_{m+1,\eps_n}\| + \|u_{m+1,\eps_n}-u_{m+1}\|, 
	\end{equation*}
	where all the norms above are $\xX(0,\tau_m \wedge \tau_{m+1})$. The first and third terms can be made arbitrarily small when $n$ is large. Also, the convergence of $u_{m,\eps_n}$ to $u_m$ and the definition of the stopping times in \eqref{eq:stopping_time}
	imply that $\tau_{m,\eps_n} \geq \tau_m$ for all sufficiently large $n$. By \eqref{eq:st_dd}, we then have $u_{m,\eps_n} = u_{m+1,\eps_n}$ in $\xX(0,\tau_{m} \wedge \tau_{m+1})$ almost surely if $n$ is large. This shows that on a set of full measure, $|u_{m} - u_{m+1}|$ can be made arbitrarily small, thus concluding the proof. 
\end{proof}

\begin{proof} [Proof of Theorem~\ref{th:main}]
	We are now ready to prove Theorem~\ref{th:main}. We first show that $\{u_{m}\}$ is Cauchy in $L_{\omega}^{\rho_0} \xX(0,1)$. To see this, we fix $\delta>0$ arbitrary. By Proposition~\ref{pr:uniform_m} and that $u_{m,\eps} \rightarrow u_{m}$ in $L_{\omega}^{\rho_0} \xX(0,1)$, we have
	\begin{equation} \label{eq:trun_small}
	\Pr \Big( \|u_{m}\|_{\xX(0,1)}^{5} \geq K \Big) \leq \frac{C(M,\rho_0)}{K^{\rho_0}}
	\end{equation}
	for every $m$ and every $K$. Now, for every $m$, $m'$ and $K$, we let
	\begin{equation*}
	\Omega_{m,m'}^{K} = \Big\{ \|u_{m}\|_{\xX(0,1)}^{5} \geq K \Big\} \cup \Big\{ \|u_{m'}\|_{\xX(0,1)}^{5} \geq K \Big\}. 
	\end{equation*}
	By H\"older inequality, the bound \eqref{eq:trun_small} and Proposition~\ref{pr:uniform_m} with $\rho=2\rho_0$, we know there exists $C>0$ such that
	\begin{equation} \label{eq:truncation_small}
	\big\|\1_{\Omega_{m,m'}^{K}} \big(u_{m} - u_{m'} \big) \big\|_{L_{\omega}^{\rho_0} \xX(0,1)} \leq \big(\Pr(\Omega_{m,m'}^{K})\big)^{\frac{1}{2}} \|u_{m}-u_{m'}\|_{L_{\omega}^{2\rho_0}\xX(0,1)} \leq \frac{C(M,\rho_0)}{K^{1/10}}
	\end{equation}
	for all $K$, $m$ and $m'$. Hence, there exists $K^*$ large enough so that the right hand side of \eqref{eq:truncation_small} is smaller than $\delta$ if $K=K^*$. Take any $m, m' > K^*+1$. By Lemma~\ref{le:st}, we have
	\begin{equation*}
	u_{m} = u_{m'} \quad \text{on} \phantom{1} (\Omega_{m,m'}^{K})^{c}. 
	\end{equation*}
	Combining with \eqref{eq:truncation_small}, we deduce that $\|u_{m} - u_{m'}\|_{L_{\omega}^{\rho_0}\xX(0,1)} < \delta$ whenever $m,m' > K^*+1$. Hence, $\{u_{m}\}$ is Cauchy and converges to a limit $u$ in $L_{\omega}^{\rho_0}\xX(0,1)$. 
	
	We now show that the limit $u$ satisfies the Duhamel's formula \eqref{eq:duhamel_main}. This part follows similarly as the proof for the equation of $u_{m}$, but only easier. We need to show that each term on the right hand side of \eqref{eq:duhamel_trun_sub} converges to the corresponding term in $L_{\omega}^{\rho_0}\xX(0,1)$ as $m \rightarrow +\infty$. The convergence of
	\begin{equation*}
	\int_{0}^{t} \sS(t-s) u_{m}(s) {\rm d} W_s \quad \text{and} \quad \int_{0}^{t} \sS(t-s) \big( F_{\Phi} u_{m}(s) \big) {\rm d}s
	\end{equation*}
	follows immediately from the bounds in Propositions~\ref{pr:pre_stochastic} and~\ref{pr:pre_correction} and that $u_{m} \rightarrow u$ in $L_{\omega}^{\rho_0}\xX(0,1)$. We now turn to the nonlinearity
	\begin{equation*}
	\int_{0}^{t} \sS(t-s) \Big( \theta_{m}\big(\|u_m\|_{\xX_2(0,s)}^{5}\big) \nN\big(u_m(s)\big) - \nN\big(u(s)\big) \Big) {\rm d}s. 
	\end{equation*}
	The integrand satisfies the pointwise bound
	\begin{equation*}
	\begin{split}
	&\phantom{111}\theta_{m}\big(\|u_m\|_{\xX_2(0,s)}^{5}\big) \nN\big(u_m(s)\big) - \nN\big(u(s)\big)\\
	&\leq \big| \theta_{m}\big(\|u_m\|_{\xX_2(0,s)}^{5}\big) -1 \big| |u_{m}(s)|^{5} + C |u_{m}(s) - u(s)| \big( |u_{m}(s)|^{4} + |u(s)|^{4} \big)\\
	&\leq \1_{\big\{\|u_m\|_{\xX(0,1)}^{5} \geq m\big\}} |u_m(s)|^{5} + C |u_{m}(s) - u(s)| \big( |u_{m}(s)|^{4} + |u(s)|^{4} \big), 
	\end{split}
	\end{equation*}
	where we have used the fact that $\theta_m(\cdot)$ is always between $0$ and $1$, and it does not equal to $1$ only if its argument is bigger than $m$. With this pointwise bound and the Strichartz estimates \eqref{eq:pre_nonlinear_critical}, we get
	\begin{equation} \label{eq:final_nonlinear}
	\begin{split}
	&\phantom{111}\Big\|\int_{0}^{t} \sS(t-s) \Big( \theta_{m}\big(\|u_m\|_{\xX_2(0,s)}^{5}\big) \nN\big(u_m(s)\big) - \nN\big(u(s)\big) \Big) {\rm d}s\Big\|_{\xX(0,1)}\\
	&\leq C \Big( \1_{\{\|u_m\|_{\xX(0,1)}^{5} \geq m\}} \|u_{m}\|_{\xX_2(0,1)}^{5} + \|u_m-u\|_{\xX_2(0,1)} \big( \|u_m\|_{\xX_2(0,1)}^{4} + \|u\|_{\xX_2(0,1)}^{4} \big) \Big). 
	\end{split}
	\end{equation}
	Now we need to use the fact that the convergence of $u_m$ to $u$ holds in $L_{\omega}^{\rho}\xX(0,1)$ for all $\rho$ (and in particular $\rho > \rho_0$). More precisely, we have
	\begin{equation} \label{eq:final_bigger_rho}
	\|u_m - u\|_{L_{\omega}^{\rho}\xX(0,1)} \rightarrow 0, \quad \|u_m\|_{L_{\omega}^{\rho}\xX(0,1)} \leq C_{\rho}, \quad \|u\|_{L_{\omega}^{\rho}\xX(0,1)} \leq C_{\rho}
	\end{equation}
	for all $\rho \geq 1$. Hence, taking $L_{\omega}^{\rho_0}$-norm on both sides of \eqref{eq:final_nonlinear}, and using H\"older inequality and \eqref{eq:final_bigger_rho}, we see that $\|\dD\|_{L_{\omega}^{\rho_0}\xX(0,1)} \rightarrow 0$ as $m \rightarrow +\infty$. This shows that $u$ satisfies the Duhamel's formula \eqref{eq:duhamel_main}. 
	
	We finally turn to the stability of the solution under perturbation of initial data. Let $u_0, v_0 \in L_{\omega}^{\infty}L_{x}^{2}$ with $\|u_0\|_{L_{\omega}^{\infty}L_{x}^{2}}, \|v_0\|_{L_{\omega}^{\infty}L_{x}^{2}} \leq M$. Let $u$ and $v$ be the two solutions constructed from the above mentioned procedure with initial data $u_0$ and $v_0$ respectively. We have
	\begin{equation} \label{eq:final_stability_triangle}
	\|u-v\|_{L_{\omega}^{\rho_0}\xX(0,1)} \leq \|u-u_m\|_{L_{\omega}^{\rho_0}\xX(0,1)} + \|u_m - v_m\|_{L_{\omega}^{\rho_0}\xX(0,1)} + \|v_m - v\|_{L_{\omega}^{\rho_0}\xX(0,1)}. 
	\end{equation}
	By the arguments above for the Cauchy property of $\{u_m\}$ and in particular the bound \eqref{eq:truncation_small}, we have
	\begin{equation*}
	\|u-u_m\|_{L_{\omega}^{\rho_0}\xX(0,1)} \leq \frac{C}{m^{1/10}}\;, \qquad \|v-v_m\|_{L_{\omega}^{\rho_0}\xX(0,1)} \leq \frac{C}{m^{1/10}}
	\end{equation*}
	for all $m$. The constant $C$ depends on $M$ and $\rho_0$ only. Hence, for every $\delta>0$, we can choose $m$ sufficiently large depending on $M$ and $\rho_0$ only such that the first and third terms on the right hand side of \eqref{eq:final_stability_triangle} are both smaller than $\frac{\delta}{3}$. By Proposition~\ref{pr:uniform_stable}, with this choice of $m$, there exists $\kappa>0$ such that if $\|u_0 - v_0\|_{L_{\omega}^{\infty}L_{x}^{2}} < \kappa$, then the second term in \eqref{eq:final_stability_triangle} is also smaller than $\frac{\delta}{3}$. Since $m$ depends on $M$ and $\rho_0$ only, so does $\kappa$. The proof is thus complete. 
\end{proof}

\appendix

\endappendix

\bibliographystyle{Martin}
\bibliography{Refs}

\end{document}